%% file: main.tex
\documentclass{siamart190516}

\input{shared}

\ifpdf
\hypersetup{
  pdftitle={Convenient geometries on correlation matrices},
  pdfauthor={Y. Thanwerdas and X. Pennec}
}
\fi

\begin{document}

\maketitle

\begin{abstract}
In contrast to SPD matrices, few tools exist to perform Riemannian statistics on the open elliptope of full-rank correlation matrices. The quotient-affine metric was recently built as the quotient of the affine-invariant metric by the congruence action of positive diagonal matrices. The space of SPD matrices had always been thought of as a Riemannian homogeneous space. In contrast, we view in this work SPD matrices as a Lie group and the affine-invariant metric as a left-invariant metric. This unexpected new viewpoint allows us to generalize the construction of the quotient-affine metric and to show that the main Riemannian operations can be computed numerically. However, the uniqueness of the Riemannian logarithm or the Fréchet mean are not ensured, which is bad for computing on the elliptope. Hence, we define three new families of Riemannian metrics on full-rank correlation matrices which provide Hadamard structures, including two flat. Thus the Riemannian logarithm and the Fréchet mean are unique. We also define a nilpotent group structure for which the affine logarithm and the group mean are unique. We provide the main Riemannian/group operations of these four structures in closed form.
\end{abstract}



\section{Introduction}

The (open) elliptope is the set of full-rank correlation matrices, it is open in the affine space of symmetric matrices with unit diagonal. Its geometry has been much less studied than the one of the cone of Symmetric Positive Definite (SPD) matrices. Nevertheless, several applications could benefit from well suited geometries on this manifold: graphical networks, brain connectomes \cite{Varoquaux10}, finance \cite{Rebonato00,Marti21} or phylogenetic trees \cite{Garba21}. A few operations have been proposed such as sampling by projecting samples on spheres \cite{Rebonato00,Kercheval08}, computing distances via the Hilbert geometry of convex sets \cite{Nielsen19} or via a recent parametrization by $\R^{n(n-1)/2}$ \cite{Archakov21}.

The open elliptope was also recently characterized as the quotient manifold of SPD matrices by the smooth, proper and free congruence action of positive diagonal matrices \cite{David19,David19-thesis}. Indeed, a correlation matrix is obtained from the covariance matrix by dividing by the standard deviation of each variable. This is equivalent to multiply the covariance matrix on left and right by the diagonal matrix of the inverse square roots of the variances. Hence all covariance matrices that are congruent up to a diagonal matrix represent the same correlation matrix. This structure allowed to quotient the well known affine-invariant metrics on SPD matrices to the so called quotient-affine metrics on full-rank correlation matrices \cite{David19,David19-thesis,Thanwerdas21-GSI}. They only differ by a scaling factor so they are often referred as \textit{the} quotient-affine metric. It offers promising perspectives since it is geodesically complete with closed form expressions for the exponential map, the Levi-Civita connection and the sectional curvature. However, it was difficult to determine the sign and potential bounds of the curvature of the quotient-affine metric. In this work, we show that the sectional curvature can take both positive and negative values, that it is bounded from below and unbounded from above. Therefore, the elliptope of full-rank correlation matrices endowed with the quotient-affine metric is neither a Hadamard space nor a $\CAT(k)$ space for any $k\in\R$.

This calls for new Riemannian metrics on the elliptope. Indeed, we could expect to find suitable metrics with non-positive or even null curvature since the elliptope is an open set of an affine space. The structure of the space is an important element of modeling because simple structures often bring good theoretical properties and better computability of the geometric operations. For example, in a Hadamard space, the Fréchet mean is unique. Recall that in a metric space $(\mc{M},d)$, a Fréchet mean of points $x_1,...,x_k\in\mc{M}$ is a point $x\in\mc{M}$ which minimizes the function $x\in\mc{M}\lmto\sum_{i=1}^kd(x,x_i)^2\gs 0$. This is the case of the affine-invariant metric on SPD matrices \cite{Skovgaard84} or the Fisher metric of beta and Dirichlet distributions \cite{LeBrigant21}. Moreover, if one has a Euclideanization of the space, i.e. a smooth diffeomorphism to a Euclidean space, then all operations become trivial. This is the case of the log-Euclidean \cite{Arsigny06} or the log-Cholesky \cite{Lin19} metrics on SPD matrices.

Another strong element of modeling is the invariance of the geometry under a given group action. Let us give some examples. First, on SPD matrices, one could require the invariance of the geometry under all affine transformations of the feature vector. Indeed in EEG, if we assume that there is an affine transformation from one brain to another at first order, then electromagnetic fields are transformed similarly since they satisfy linear equations. The affine-invariant metric significantly improved the results of classification in BCI \cite{Barachant13}. Second, one could want the analysis to be invariant from the scale of each variable. As explained previously, this corresponds to the invariance under the congruence action of positive diagonal matrices and it suggests to focus more on correlation matrices than on covariance matrices. Third, all the operations mentioned above on correlation matrices are invariant under permutations. It means that the statistical analyses are invariant under any joint permutation of the rows and columns of the correlation matrices in the dataset. It is an advantage if the way the variables (or channels) are ordered is arbitrary. It is a drawback if the order is chosen for a certain reason. For instance, in the auto-correlation matrix of a signal, the variables represent different times, which are not exchangeable. One could want a structure on correlation matrices that respects the time structure instead of being invariant under permutations. In this work, we focus on non-permutation-invariant geometric structures.

Moreover, an invariance may also lead to simple structures with good properties. For example, a Riemannian homogeneous space is geodesically complete, which is well suited for interpolation and extrapolation. The geodesics of naturally reductive homogeneous spaces are the orbits of the one-parameter subgroups so they are known in closed form. In a Riemannian symmetric space, the parallel transport is obtained in closed form by composition of two symmetries. So it is another good reason to study Lie group actions on our space.

\subsection{Overview of the results}

In this work, we define new non-permutation-invariant geometric structures on full-rank correlation matrices following two directions. Firstly, we define a non-permutation-invariant generalization of quotient-affine metrics. Our approach consists in studying the congruence action of several matrix Lie groups on SPD matrices. Our main result is a characterization of affine-invariant metrics, i.e. metrics that are invariant under the congruence action of the general linear group $\GL(n)$, by the joint invariance of a pair of subgroups. These are the group of permutation matrices $\mf{S}(n)$ and the group of lower triangular matrices with positive diagonal $\LT^+(n)$. In other words, the affine-invariant metrics are the unique ($\mf{S}(n)\times\LT^+(n)$)-invariant metrics on SPD matrices. Therefore, the family of $\LT^+(n)$-invariant metrics appears as a natural non-permutation-invariant generalization of affine-invariant metrics. Moreover, we show that such metrics are exactly the pullback metrics of left-invariant metrics on the Lie group $\LT^+(n)$ by the Cholesky map, so we call them Lie-Cholesky metrics. In particular, there exists a Lie group structure on SPD matrices (namely, the product of the Cholesky factors) such that affine-invariant metrics are left-invariant metrics on that Lie group. Finally, since the Lie group $\LT^+(n)$ contains positive diagonal matrices, Lie-Cholesky metrics descend to quotient metrics on full-rank correlation matrices by the same procedure as affine-invariant metrics. We show that several Riemannian operations are numerically computable such as the Riemannian metric, the exponential map, the logarithm map or the Riemannian distance. However, despite the nice theoretical results and the computability of the main geometric operations, quotient-Lie-Cholesky metrics don't have obvious nice geometric properties. For example, it is not clear whether the Riemannian logarithm and the Fréchet mean are unique or not.

Hence, secondly, to solve the drawbacks of quotient-affine and quotient-Lie-Cho\-lesky metrics, we propose a series of new geometries on full-rank correlation matrices for which the mean of finite samples is unique. They are built in a different way than the Lie-Cholesky metrics, the only common point being the use of the Cholesky map to convey the structures from triangular matrices to correlation matrices. Hence they are not permutation-invariant either for $n\gs 3$. We define the poly-hyperbolic-Cholesky metrics as the pullbacks of weighted products of hyperbolic spaces $\Hyp^1\times\cdots\times\Hyp^{n-1}$. The metrics of this family provide a structure of Riemannian symmetric space with non-positive bounded curvature, thus Hadamard. All operations are known in closed form. We define the Euclidean-Cholesky and the log-Euclidean-Cholesky metrics as the pullbacks by two maps derived from the Cholesky map of Euclidean structures on the vector space $\LT^0(n)$ of strictly lower triangular matrices. The metrics in these two families are flat and geodesically complete so all Riemannian operations are known in closed form, the former being less costly than the latter. Finally, we define a Lie group structure as the pullback of the Lie group (for matrix multiplication) $\LT^1(n)$ of lower triangular matrices with unit diagonal. Here, we consider the geometric structure induced by the canonical Cartan-Schouten connection rather than by a Riemannian metric. Indeed, the (group) exponential map allows to define a group mean, which is unique here because the Lie algebra of the Lie group is nilpotent \cite{Buser81}. Finally, in dimension 2, we show that on the one hand, the quotient-affine and the poly-hyperbolic-Cholesky metrics coincide, and on the other hand, the Euclidean-Cholesky and the log-Euclidean-Cholesky metrics coincide and their geodesics coincide with the group geodesics. In particular, these geodesics provide a new interpolation of the correlation coefficient, different than the one proposed in \cite{Thanwerdas21-GSI}.

In the remainder of this section, we introduce the matrix notations, the notions of covariance and correlation matrices and the Cholesky map. In Section 2, we recall the definition of the quotient-affine metrics and we prove that its sectional curvature takes both negative and positive values and is unbounded from above. In Section 3, we give a characterization of affine-invariant metrics in function of the congruence action of other groups. This allows us to introduce Lie-Cholesky metrics on SPD matrices and quotient-Lie-Cholesky metrics on full-rank correlation matrices. In Section 4, we introduce four new non-permutation-invariant structures on full-rank correlation matrices for which the mean is unique. We conclude in Section 5.

\subsection{Concepts and notations}

\subsubsection{Matrix notations}

Our main matrix space notations are given in Table \ref{tab:matrix}. In the paper, we also use the following linear maps:
\begin{enumerate}[label=$\cdot$]
    \itemsep0em
    \item $(A,B)\in\Mat(n)\times\Mat(n)\lmto A\bullet B=[A_{ij}B_{ij}]_{1\ls i,j\ls n}\in\Mat(n)$ is the Hadamard product of matrices,
    \item $\Diag:M\in\Mat(n)\lmto[\delta_{ij}M_{ij}]_{1\ls i,j\ls n}\in\Diag(n)$ selects the diagonal terms,
    \item $\off=\Id_{\Mat(n)}-\Diag:\Mat(n)\lto\ker\Diag$ selects the off-diagonal terms,
    \item $\low:M\in\Mat(n)\lmto[\delta_{i\gs j}M_{ij}]_{1\ls i,j\ls n}\in\LT(n)$ selects the lower triangular terms, \textit{including} the diagonal terms,
    \item $\low_0=\low-\Diag:\Mat(n)\lto\LT^0(n)$ selects the strictly lower triangular terms, \textit{excluding} the diagonal terms,
    \item $\low_\S=\low_0+\frac{1}{2}\Diag:\Sym(n)\lto\LT(n)$ selects the strictly lower triangular terms and half of the diagonal terms, so that when a symmetric matrix writes $M=L+L^\top\in\Sym(n)$ with $L\in\LT(n)$, then $L=\low_\S(M)$,
    \item $\somme:M\in\Mat(n)\lmto\sum_{i,j}M_{ij}=\mathds{1}^\top M\mathds{1}$ sums the terms of the matrix,
    \item $\somme:x\in\R^n\lmto\sum_ix_i=\mathds{1}^\top x$ sums the terms of the vector,
\end{enumerate}
where $\mathds{1}=(1,...,1)^\top\in\R^n$.

\begin{table}[h]
    \centering
    \begin{tabular}{|c|c|c|c|}
    \hline
    \multicolumn{2}{|c|}{Matrix vector spaces} & \multicolumn{2}{c|}{Matrix manifolds}\\
    \hline
    $\Mat(n)$ & Squared & $\GL(n)$ & General linear group  \\
    \hline
    \multirow{2}{*}{$\Skew(n)$} & \multirow{2}{*}{Skew-symmetric} & $\Orth(n)$ & Orthogonal group\\
    \cline{3-4}
    && $\SO(n)$ & Special orthogonal group\\
    \hline
    $\Sym(n)$ & Symmetric & $\Sym^+(n)$ & Sym Positive Definite cone\\
    \hline
    $\Hol(n)$ & Symmetric hollow & $\Cor^+(n)$ & Full-rank correlation elliptope\\
    \hline
    $\LT(n)$ & Lower Triangular & $\LT^+(n)$ & LT with positive diagonal\\
    \hline
    $\LT^0(n)$ & LT with null diagonal & $\LT^1(n)$ & LT with unit diagonal\\
    \hline
    $\Diag(n)$ & Diagonal & $\Diag^+(n)$ & Positive diagonal group\\
    \hline
    \end{tabular}
    \caption{Matrix space notations}
    \label{tab:matrix}
\end{table}

The possibly uncommon notations in Table \ref{tab:matrix} are:
\begin{enumerate}[label=$\cdot$]
    \itemsep0em
    \item $\Sym^+(n)=\{\Sigma\in\Sym(n)|\Sigma>0\}$ where $>$ is the Loewner order,
    \item $\Cor^+(n)=\{\Sigma\in\Sym^+(n)|\Diag(\Sigma)=I_n\}$,
    \item $\Hol(n)=\{M\in\Sym(n)|\Diag(M)=0\}$ where ``hollow" means vanishing diagonal,
    \item $\LT(n)=\{\low(M)|M\in\Mat(n)\}$,
    \item $\LT^+(n)=\{L\in\LT(n)|\Diag(L)\in\Diag^+(n)\}$,
    \item $\LT^0(n)=\{L\in\LT(n)|\Diag(L)=0\}$,
    \item $\LT^1(n)=\{L\in\LT(n)|\Diag(L)=I_n\}=I_n+\LT^0(n)$.
\end{enumerate}

We denote $\mf{S}(n)$ the permutation group of order $n$.

\subsubsection{Covariance and correlation matrices}

Given an invertible covariance matrix $\Sigma=(\Cov(X_i,X_j))_{1\ls i,j\ls n}\in\Sym^+(n)$ of a random vector $X$, the corresponding correlation matrix is defined by $C=\cor(\Sigma)=(\Cor(X_i,X_j))_{1\ls i,j\ls n}$ where:
\begin{align}
    \Cov(X_i,X_j)&=\mathbb{E}(X_iX_j)-\mathbb{E}(X_i)\mathbb{E}(X_j),\\
    \Cor(X_i,X_j)&=\frac{\Cov(X_i,X_j)}{\sqrt{\Cov(X_i,X_i)}\sqrt{\Cov(X_j,X_j)}}=\frac{\Sigma_{ij}}{\sqrt{\Sigma_{ii}}\sqrt{\Sigma_{jj}}}\\
    &=[\Diag(\Sigma)^{-1/2}\,\Sigma\,\Diag(\Sigma)^{-1/2}]_{ij},
\end{align}
with $\Diag(\Sigma)=\diag(\Sigma_{11},...,\Sigma_{nn})\in\Diag^+(n)$. Moreover, if $D\in\Diag^+(n)$, then the correlation matrix associated to $D\Sigma D$ is again $\cor(\Sigma)$. It is known as the invariance of the correlation matrix under the scaling of each component of the random vector. In other words, the surjective map $\cor:\Sigma\in\Sym^+(n)\lmto\cor(\Sigma)\in\Cor^+(n)$ is invariant under the group action $(D,\Sigma)\in\Diag^+(n)\times\Sym^+(n)\lmto D\Sigma D\in\Sym^+(n)$ and the orbit space $\Sym^+(n)/\Diag^+(n)$ can be identified with $\Cor^+(n)$. Note that, via this identification, the induced topology $\Cor^+(n)\hookrightarrow\Sym(n)$ coincides with the quotient topology of $\Sym^+(n)/\Diag^+(n)$. The action being proper \cite{David19}, the quotient manifold theorem \cite{Lee12} states that the orbit space $\Sym^+(n)/\Diag^+(n)$ has a unique smooth manifold structure such that the canonical surjection $\Sym^+(n)\lto\Sym^+(n)/\Diag^+(n)$ is a smooth submersion. Since $\cor:\Sym^+(n)\lto\Cor^+(n)$ is a smooth submersion and since $\Cor^+(n)$ is $\Sym^+(n)/\Diag^+(n)$ as a topological space, this smooth structure coincides with the induced smooth structure of $\Cor^+(n)$.

\subsubsection{The Cholesky map}

A Cholesky decomposition of a symmetric positive semi-definite matrix $\Sigma$ is a factorization of the form $\Sigma=LL^\top$ where $L\in\LT(n)$ is a lower triangular matrix. If $\Sigma$ is positive definite, then there exists a unique triangular matrix with positive diagonal $L\in\LT^+(n)$ such that $\Sigma=LL^\top$. This allows to define the Cholesky bijective map:
\begin{equation}
	\Chol:\Sigma\in\Sym^+(n)\lmto L\in\LT^+(n),
\end{equation}
whose inverse is the smooth map $\phi:L\in\LT^+(n)\lmto LL^\top\in\Sym^+(n)$. Moreover, the Cholesky map is smooth. Indeed, it is the product of two smooth maps $\Chol(\Sigma)=L(\Sigma)\sqrt{D(\Sigma)}$ where $L\equiv L(\Sigma)\in\LT^1(n)$ and $D\equiv D(\Sigma)\in\Diag^+(n)$ are recursively defined for all $(i,j)\in\{1,...,n\}^2$ with $i>j$ by:
\begin{align}
	D_{ii}&=\Sigma_{ii}-\sum_{j=1}^{i-1}{L_{ij}^2D_{jj}}>0,\\
	L_{ij}&=\frac{1}{D_{jj}}\left(\Sigma_{ij}-\sum_{k=1}^{j-1}L_{ik}L_{jk}D_{kk}\right),
\end{align}
the order of the computations being $\{D_{11}\}\to\cdots\to\{L_{i1}\to L_{i2}\to\cdots\to L_{i,i-1}\to D_{ii}\}\to\cdots\to\{L_{n1}\to\cdots\to L_{n,n-1}\to D_{nn}\}$.

We denote:
\begin{equation*}
    \mc{L}=\Chol(\Cor^+(n))=\left\{L=\begin{pmatrix}L_1\\\vdots\\L_n\end{pmatrix}\in\LT^+(n)|\forall i\in\{1,...,n\},\|L_i\|^2=L_iL_i^\top=1\right\},
\end{equation*} which is diffeomorphic to $\Cor^+(n)$ via the Cholesky map. We also define a Cholesky-based diffeomorphism between $\Cor^+(n)$ and $\LT^1(n)$:
\begin{align}
    \Theta&:C\in\Cor^+(n)\lmto\Gamma=\Diag(\Chol(C))^{-1}\Chol(C)\in\LT^1(n),\\
    \Phi&:\Gamma\in\LT^1(n)\lmto C=\Diag(\Gamma\Gamma^\top)^{-1/2}\Gamma\Gamma^\top\Diag(\Gamma\Gamma^\top)^{-1/2}\in\Cor^+(n).
\end{align}
We clearly have the relations $\Phi=\Theta^{-1}=\cor\circ\phi$.

We compute the differentials of $\phi:\LT^+(n)\lto\Sym^+(n)$ and $\Chol:\Sym^+(n)\lto\LT^+(n)$. For all $Z\in T_L\LT^+(n)\simeq\LT(n)$:
\begin{align}
	V:=d_L\phi(Z)&=ZL^\top+LZ^\top,\\
	L^{-1}VL^{-\top}&=L^{-1}Z+(L^{-1}Z)^\top,\nonumber\\
	\low_{\S}(L^{-1}VL^{-\top})&=L^{-1}Z,\nonumber\\
	Z=d_\Sigma\Chol(V)&=L\,\low_{\S}(L^{-1}VL^{-\top}). \label{eq:diff_Chol}
\end{align}
In particular, $d_{I_n}\Chol=\low_\S$.

\section{Quotient-affine metrics}

In this section, we briefly recall how to build quotient metrics to fix the notations (Section 2.1), then we recall the definition of quotient-affine metrics (Section 2.2) and finally we show that the sectional curvature takes both negative and positive curvature and is unbounded from above (Section 2.3).

\subsection{Quotient metrics}

Given smooth manifolds $\mc{M},\mc{M}'$ and a smooth submersion $\pi:\mc{M}\lto\mc{M}'$, we can define the vertical space $\mc{V}_x=\ker d_x\pi\subset T_x\mc{M}$, where $d_x\pi:T_x\mc{M}\lto T_{\pi(x)}\mc{M}'$ is the differential of the map $\pi$ at point $x\in\mc{M}$. The horizontal space $\mc{H}_x$ can be any supplementary vector space, i.e. such that $\mc{V}_x\oplus\mc{H}_x=T_x\mc{M}$. Given a horizontal distribution $x\lmto\mc{H}_x$, there exist vertical and horizontal projections $\ver_x:T_x\mc{M}\lto\mc{V}_x$ and $\hor_x:T_x\mc{M}\lto\mc{H}_x$. Moreover, the linear map $(d_x\pi)_{|\mc{H}_x}:\mc{H}_x\lto T_{\pi(x)}\mc{M}'$ is an isomorphism. Its inverse isomorphism is called the horizontal lift and denoted $\#_x:X\in T_{\pi(x)}\mc{M}'\lto X^\#_x\in\mc{H}_x$. When $\mc{M}$ is endowed with a Riemannian metric, there is a canonical choice of horizontal space which is the orthogonal $\mc{H}_x=\mc{V}_x^\perp$. In this case, the projections are orthogonal. 

Given a smooth manifold $\mc{M}$ on which a Lie group $G$ acts smoothly, properly and freely, the quotient space $\mc{M}/G$ admits a unique smooth manifold structure that turns the canonical projection $\pi:\mc{M}\lto\mc{M}/G$ into a smooth submersion \cite[Theorem 21.10]{Lee12}.
The vertical distribution is $G$-equivariant, i.e. $\mc{V}_{a\cdot x}=a\cdot\mc{V}_x$ for all $a\in G$ (where $\cdot$ is the group action on $\mc{M}$ and $T\mc{M}$). Given a $G$-invariant Riemannian metric $g$ on $\mc{M}$, the horizontal distribution is $G$-equivariant and the metric descends to a metric $g'$ on $\mc{M}/G$ defined by $g'_{\pi(x)}(X,X)=g_x(X^\#_x,X^\#_x)$.

\subsection{Definition of quotient-affine metrics}

Applying this to $\mc{M}=\Sym^+(n)$, $G=\Diag^+(n)$ and $\mc{M}/G\simeq\Cor^+(n)$, the submersion $\cor:\Sigma\in\Sym^+(n)\lmto\Diag(\Sigma)^{-1/2}\,\Sigma\,\Diag(\Sigma)^{-1/2}\in\Cor^+(n)$ \cite{David19,David19-thesis} allows to descend any $\Diag^+(n)$-invariant Riemannian metric on $\Sym^+(n)$ to a Riemannian metric on $\Cor^+(n)$ \cite{ONeill66}.

A natural example of $\Diag^+(n)$-invariant Riemannian metric on SPD matrices is provided by the affine-invariant metric defined for all $\Sigma\in\Sym^+(n)$ and $V\in T_\Sigma\Sym^+(n)\simeq\Sym(n)$ by:
\begin{equation}
    g^{\mathrm{AI}(\alpha,\beta)}_\Sigma(V,V)=\alpha\,\tr(\Sigma^{-1}V\Sigma^{-1}V)+\beta\,\tr(\Sigma^{-1}V)^2 \label{eq:affine-invariant}
\end{equation}
where $\alpha>0$ and $\beta>-\frac{\alpha}{n}$. Its sectional curvature is for $V,W\in\Sym(n)$ \cite{Skovgaard84,Thanwerdas21-LAA}:
\begin{align}
    \kappa^{\mathrm{AI}(\alpha,\beta)}_\Sigma(V,W)&=\frac{1}{4\alpha}\tr((\Sigma^{-1}V\Sigma^{-1}W-\Sigma^{-1}W\Sigma^{-1}V)^2)\in\left[-\frac{1}{2\alpha};0\right]
\end{align}

The quotient-affine metric is the quotient of the affine-invariant metric via the submersion $\cor:\Sym^+(n)\lto\Cor^+(n)$. It does not depend on $\beta$ and it writes $\alpha\,g^\QA$ with, for all $C\in\Cor^+(n)$ and $X\in T_C\Cor^+(n)\simeq\Hol(n)$:
\begin{equation}
    g^\QA_C(X,X)=\tr((C^{-1}X)^2)-2\mathds{1}^\top\Diag(C^{-1}X)(I_n+C\bullet C^{-1})^{-1}\Diag(C^{-1}X)\mathds{1}.
\end{equation}
Note that it is invariant under permutations.

The vertical/horizontal distributions/projections and the horizontal lift are for all $\Sigma\in\Sym^+(n)$, $V\in T_\Sigma\Sym^+(n)\simeq\Sym(n)$ and $X\in T_{\cor(\Sigma)}\Cor^+(n)\simeq\Hol(n)$ \cite{Thanwerdas21-GSI}:
\begin{enumerate}[label=$\cdot$]
    \itemsep0em
    \item $\mc{V}_\Sigma=\{D\Sigma+\Sigma D|D\in\Diag(n)\}$,
    \item $\mc{H}_\Sigma=\{V\in\Sym(n)|\Sigma^{-1}V+V\Sigma^{-1}\in\Hol(n)\}=\mc{S}_{\Sigma^{-1}}(\Hol(n))$,
    \item $\ver_\Sigma(V)=D\Sigma+\Sigma D\in\mc{V}_\Sigma$ with $D=\diag\left((I_n+\Sigma\bullet\Sigma^{-1})^{-1}\Diag(\Sigma^{-1}V)\mathds{1}\right)$,
    \item $\hor_\Sigma(V)=V-\ver_\Sigma(V)\in\mc{H}_\Sigma$,
    \item $X^\#_\Sigma=\hor_\Sigma(\Diag(\Sigma)^{1/2}X\Diag(\Sigma)^{1/2})\in\mc{H}_\Sigma$,
\end{enumerate}
where $\mc{S}_A(X)$ is the unique solution of the Sylvester equation $A\mc{S}_A(X)+\mc{S}_A(X)A=X$.

These operations allow to write the sectional curvature for all $C\in\Cor^+(n)$ and $X,Y\in T_C\Cor^+(n)\simeq\Hol(n)$ \cite{Thanwerdas21-GSI}:
\begin{equation}
    \kappa^\QA_C(X,Y)=\underset{\in\left[-\frac{1}{2};0\right]}{\underbrace{\kappa^{\mathrm{AI}}_C(X^\#_C,Y^\#_C)}}+\underset{\gs 0}{\underbrace{\frac{3}{8}\frac{\mu^\top(I_n+C\bullet C^{-1})^{-1}\mu}{g^\QA_C(X,X)g^\QA_C(Y,Y)-g^\QA_C(X,Y)^2}}},
\end{equation}
with $\mu\in\R^n$ defined by $\mu=[D(X,Y)-D(Y,X)]\mathds{1}\in\R^n$ and $D(X,Y)\in\Diag(n)$ defined by $D(X,Y)=\Diag([C^{-1}\Diag(X^\#_C)C,C^{-1}Y^\#_C])$, where $[A,B]=AB-BA$ is the commutator of squared matrices.

\subsection{Complement: bounds of curvature}

\begin{theorem}\label{thm:bounds_curvature}
The sectional curvature of the quotient-affine metric takes positive and negative values. It is bounded from below and unbounded from above.
\end{theorem}

\begin{proof}
First of all, $\kappa_C(X,Y)\gs\kappa^{\mathrm{AI}}_C(X,Y)\gs-\frac{1}{2}$ so the curvature is bounded from below. Second, at $C=I_n$, $X^\#=X$ and $Y^\#=Y$ so $\Diag(X^\#)=\Diag(Y^\#)=0$ and $\mu=0$. Hence, $\kappa_{I_n}(X,Y)=\kappa^{\mathrm{AI}}_{I_n}(X,Y)\ls 0$ and for example $\kappa_{I_n}(E_{ij},E_{ik})=-\frac{1}{8}<0$ with $i\ne j\ne k\ne i\in\{1,...,n\}$ \cite{Thanwerdas21-LAA}. So the curvature takes negative values.
Third, let $X=\mathds{11}^\top-I_n$ and $Y=\mu\mathds{1}^\top+\mathds{1}\mu^\top-2\,\diag(\mu)$ with $\somme(\mu)=\mathds{1}^\top\mu=0$ where $\mu\in\R^n$. Let $C=(1-\rho)I_n+\rho\mathds{11}^\top\in\Cor^+(n)$ for $\rho\in(-\frac{1}{n-1},1)$. We show in the supplementary material that $\kappa_C(X,Y)$ tends to $+\infty$ when $\rho\to-\frac{1}{n-1}$, which proves that the curvature takes positive values and it is not bounded from above.
\end{proof}

Hence, the Riemannian manifold of full-rank correlation matrices endowed with the quotient-affine metric is geodesically complete but it is not a $\mathrm{CAT}(k)$ space for any $k\in\R$. In particular, it is not a Hadamard space as one could have hoped. It may be a problem for many algorithms because the Riemannian logarithm and the Fréchet mean are not ensured to be unique.

In Section 3, we relax the invariance under permutations that might be unsuitable in some contexts. In Section 4, we define non-permutation-invariant metrics which in addition bring uniqueness of the mean.

\section{Generalization of quotient-affine metrics}

In this section, we define a family of quotient metrics which generalize quotient-affine metrics and which are not invariant by permutations. In Section 3.1, we start by giving an overview of the congruence action of several subgroups of the general linear group $\GL(n)$ on SPD matrices. In particular, we show that affine-invariant metrics are exactly ($\mf{S}(n)\times\LT^+(n)$)-invariant metrics so the family of $\LT^+(n)$-invariant metrics appears as a natural generalization of affine-invariant metrics. In Section 3.2, we show that $\LT^+(n)$-invariant metrics are exactly pullbacks by the Cholesky map of left-invariant metrics on the Lie group $\LT^+(n)$. We call them Lie-Cholesky metrics. In particular, we define a Lie group structure on SPD matrices such that the affine-invariant metric is a left-invariant metric on that Lie group. This is an unexpected result as the SPD cone endowed with an affine-invariant metric is always seen as a Riemannian homogeneous (symmetric) space. In Section 3.3, we define quotient-Lie-Cholesky metrics and we show that we can compute numerically the exponential map, the logarithm map and the Riemannian distance.

\subsection{Congruence actions of matrix Lie groups on SPD matrices}

The action of congruence of the real general linear group $\GL(n)$ on SPD matrices is:
\begin{equation}
    \fun{\GL(n)\times\Sym^+(n)}{\Sym^+(n)}{(A,\Sigma)}{A\Sigma A^\top}.
\end{equation}
In the following theorem, we focus on the subactions given by the following subgroups.
\begin{enumerate}
    \itemsep0em
    \item The group of matrices with positive determinant $\GL^+(n)$.
    \item The special linear group $\SL(n)=\{A\in\GL(n)|\det(A)=1\}\subset\GL^+(n)$, which is interesting for the invariance of covariance matrices under volume-preserving linear transformations of the feature vector.
    \item The orthogonal group $\Orth(n)=\{R\in\GL(n)|RR^\top=I_n\}$ and the special orthogonal group $\SO(n)=\Orth(n)\cap\SL(n)$, which are interesting for the invariance of covariance matrices under rotations and symmetries.
    \item The group of lower triangular matrices with positive diagonal $\LT^+(n)$, which is interesting because an $\LT^+(n)$-invariant metric on SPD matrices is a left-invariant metric on a Lie group (cf. Section 3.2).
    \item The positive diagonal group $\Diag^+(n)\subset\LT^+(n)$, which is interesting for the invariance of covariance matrices under scalings on each variable.
    \item The group of positive real numbers $\R^+$ which injects itself into the general linear group via the map $\lambda\in\R^+\lmto\lambda I_n\in\GL^+(n)$. It is interesting for the invariance of covariance matrices under global scaling.
    \item The permutation group $\mf{S}(n)$ which injects itself into the orthogonal group via the map $\sigma\in\mf{S}(n)\hookrightarrow P_\sigma\in\Orth(n)$ defined by $[P_\sigma]_{ij}=\delta_{i,\sigma(j)}$ for all $i,j\in\{1,...,n\}$. It is interesting for the invariance of covariance matrices under permutation of the axes.
\end{enumerate}

In the proof, we also use the following notations: $\R^*$ is the group of invertible scalar matrices, $\Diag^*(n)$ is the group of invertible diagonal matrices, $\LT^*(n)$ is the group of invertible lower triangular matrices, $\UT^+(n)$ is the group of upper triangular matrices with positive diagonal.

\begin{theorem}[Characterization of affine-invariant metrics]
Let $g$ be a Riemannian metric on SPD matrices. The following statements are equivalent:
\begin{enumerate}
    \itemsep0em
    \item $g$ is $\GL(n)$-invariant, \label{enum:GL}
    \item $g$ is $\GL^+(n)$-invariant, \label{enum:GL+}
    \item $g$ is $\SL(n)$-invariant and $\R^+$-invariant, \label{enum:SL/R}
    \item $g$ is $\SO(n)$-invariant and $\Diag^+(n)$-invariant, \label{enum:SO/Diag}
    \item $g$ is $\mf{S}(n)$-invariant and $\LT^+(n)$-invariant. \label{enum:Sn/LT}
\end{enumerate}
\end{theorem}

\begin{proof}
The first statement clearly implies the others. To prove \ref{enum:GL+} $\nec$ \ref{enum:GL}, we need to take a general $\GL^+(n)$-invariant metric on SPD matrices and prove that it is $\GL(n)$-invariant. It amounts to prove that $\SO(n)$-invariant inner products on symmetric matrices are $\Orth(n)$-invariant, which is well known so \ref{enum:GL+} $\nec$ \ref{enum:GL}. To prove \ref{enum:SL/R} $\nec$ \ref{enum:GL+}; \ref{enum:SO/Diag} $\nec$ \ref{enum:GL+} and \ref{enum:Sn/LT} $\nec$ \ref{enum:GL}, it suffices to prove that the pairs of groups respectively generate $\GL^+(n)$, $\GL^+(n)$ and $\GL(n)$. 

The group generated by $\SL(n)$ and $\R^+$ is $\GL^+(n)$ so \ref{enum:SL/R} $\nec$ \ref{enum:GL+}. The group generated by $\SO(n)$ and $\Diag^+(n)$ is also $\GL^+(n)$ (it is clearly included in $\GL^+(n)$ and conversely it contains $\Sym^+(n)$ by the spectral theorem and $\GL^+(n)$ by polar decomposition) so \ref{enum:SO/Diag} $\nec$ \ref{enum:GL+}.

To prove that \ref{enum:Sn/LT} $\nec$ \ref{enum:GL}, let us show that the group generated by $\mf{S}(n)$ and $\LT^+(n)$ is $\GL(n)$. First, the LU decomposition exists for any square matrix modulo a permutation. More precisely, for all $A\in\GL(n)$, there exists a permutation $\sigma\in\mf{S}(n)$, a lower triangular matrix $L\in\LT^*(n)$ and an upper triangular matrix $U\in\UT^+(n)$ with $\Diag(U)=I_n$ such that $A=P_\sigma LU$. If we permute the rows and columns of $U$ according to the permutation $\sigma_0:k\lmto n+1-k$, then $P_{\sigma_0}UP_{\sigma_0}^\top$ is lower triangular with ones on the diagonal so it is in $\LT^+(n)$. Hence the group generated by $\mf{S}(n)$ and $\LT^*(n)$ is $\GL(n)$. Since $\LT^*(n)$ is generated by $\LT^+(n)$ and $\Diag^*(n)$, it suffices to prove that $\Diag^*(n)$ is generated by $\mf{S}(n)$ and $\LT^+(n)$. Since $\Diag^*(n)$ is generated by $\Diag^+(n)$ and matrices of the form $\diag(\pm 1,...,\pm 1)$ and since $\Diag^+(n)\subset\LT^+(n)$, it suffices to prove that matrices of the form $\diag(\pm 1,...,\pm 1)$ are generated by $\mf{S}(n)$ and $\LT^+(n)$. By matrix product and permutations, it suffices to prove that $\diag(-1,1,...,1)$ is generated by $\mf{S}(n)$ and $\LT^+(n)$. The following product of matrices:
\begin{equation}
\begin{array}{ccccccc}
    &&\overset{\in\LT^+(2)}{\overbrace{\begin{pmatrix}1&0\\-1&1\end{pmatrix}}}
    &&\overset{\in\UT^+(2)}{\overbrace{\begin{pmatrix}1&1\\0&1\end{pmatrix}}}
    &&\overset{\in\LT^+(2)}{\overbrace{\begin{pmatrix}1&0\\-1&1\end{pmatrix}}}\\
    & \nearrow & \downarrow & \nearrow & \downarrow & \nearrow & \downarrow\,=\\
    \mf{S}(2)\ni\begin{pmatrix}0&1\\1&0\end{pmatrix}
    &&\begin{pmatrix}-1&1\\1&0\end{pmatrix}
    &&\begin{pmatrix}-1&0\\1&1\end{pmatrix}
    &&\begin{pmatrix}-1&0\\0&1\end{pmatrix}
\end{array}
\end{equation}
can be generalized in dimension $n$ by adding a diagonal block $I_{n-2}$. Hence, the matrix $\diag(-1,1,...,1)$ is generated by $\mf{S}(n)$ and $\LT^+(n)$ so $\LT^*(n)$ as well and finally $\GL(n)$ entirely.
\end{proof}

One interpretation of this result is that these pairs of invariance ($\SL(n)$ and $\R^+$; $\SO(n)$ and $\Diag^+(n)$; $\mf{S}(n)$ and $\LT^+(n)$) are incompatible except in the affine-invariant metrics. In other words, the family of $\LT^+(n)$-invariant metrics is a natural non-permutation-invariant extension of the family of affine-invariant metrics. In the next section, we characterize them as pullback metrics of left-invariant metrics on the Lie group $\LT^+(n)$ by the Cholesky map.

\subsection{Lie-Cholesky metrics}

\subsubsection{$\LT^+(n)$-invariant metrics are Lie-Cholesky metrics}

The manifold $\LT^+(n)$ is an open set of the vector space $\LT(n)$. It is a Lie group where the internal law is the matrix multiplication.

\begin{definition}[Lie-Cholesky metrics on $\Sym^+(n)$]
A Lie-Cholesky metric on $\Sym^+(n)$ is the pullback of a left-invariant metric on the Lie group $(\LT^+(n),\times)$ via the Cholesky diffeomorphism. It is geodesically complete.
\end{definition}

\begin{theorem}[Lie-Cholesky are $\LT^+(n)$-invariant metrics]
A Riemannian metric on $\Sym^+(n)$ is a Lie-Cholesky metric if and only if it is $\LT^+(n)$-invariant.
\end{theorem}

\begin{proof}
Let $g$ be a metric on $\Sym^+(n)$. Let $\Sigma\in\Sym^+(n)$, $V\in T_\Sigma\Sym^+(n)$, $L=\Chol(\Sigma)$ and $Z=d_\Sigma\Chol(X)$. Note that Equation (\ref{eq:diff_Chol}) rewrites $L^{-1}Z=d_{I_n}\Chol(L^{-1}VL^{-\top})$. Thus, by definition of the pushforward, we have the following equalities:
\begin{align}
    (\Chol_*g)_L(Z,Z)&=g_\Sigma(V,V)\\
    (\Chol_*g)_{I_n}(L^{-1}Z,L^{-1}Z)&=g_{I_n}(L^{-1}VL^{-\top},L^{-1}VL^{-\top})
\end{align}
Hence, $g$ is a Lie-Cholesky metric if and only if $\Chol_*g$ is left-invariant if and only if the left terms are equal if and only if the right terms are equal if and only if $g$ is invariant under the action of $\LT^+(n)$.
\end{proof}

\subsubsection{Consequences}

\begin{corollary}
The Riemannian metrics on $\Sym^+(n)$ which are $\LT^+(n)$-inva\-riant are geodesically complete.
\end{corollary}

\begin{corollary}
Permutation-invariant Lie-Cholesky metrics are affine-invari\-ant metrics.
\end{corollary}

\begin{corollary}[Affine-invariant metrics on SPD matrices are left-invariant metrics on a Lie group]
The manifold $\Sym^+(n)$ has a structure of Lie group given by the matrix multiplication of the Cholesky factors, i.e. for all $\Sigma,\Sigma'\in\Sym^+(n)$, denoting $L=\Chol(\Sigma)$ and $L'=\Chol(\Sigma')$, the internal law is:
\begin{equation}
    \Sigma\star\Sigma'=(LL')(LL')^\top=L\Sigma'L^\top. \label{eq:law_group}
\end{equation}
The affine-invariant metrics $g_\Sigma(X,X)=\alpha\,\tr(\Sigma^{-1}X\Sigma^{-1}X)+\beta\,\tr(\Sigma^{-1}X)^2$ are left-invariant metrics for this Lie group structure. They are pullbacks of the left-invariant metrics on the Lie group $\LT^+(n)$ characterized by the following inner products at $I_n\in\LT^+(n)$, where $Z\in T_{I_n}\LT^+(n)\simeq\LT(n)$:
\begin{align}
	\dotprod{Z}{Z}_{\alpha,\beta}&=g_{I_n}(Z+Z^\top,Z+Z^\top) \nonumber\\
	&=\alpha\,\tr((Z+Z^\top)(Z+Z^\top))+\beta\,\tr(Z+Z^\top)^2 \nonumber\\
	&=2\alpha\,\tr(ZZ^\top)+2\alpha\,\tr(\Diag(Z)^2)+4\beta\,\tr(Z)^2.
\end{align}
We denote $\dotprod{\cdot}{\cdot}=\dotprod{\cdot}{\cdot}_{1/2,0}$, i.e. $\dotprod{Z}{Z'}=\tr(ZZ^\top)+\tr(\Diag(Z)^2)$ for $Z\in\LT(n)$.
\end{corollary}

The left-invariant metrics on $\LT^+(n)$ are parametrized by the inner products at $I_n$. Hence, they can be parameterized by self-adjoint positive definite linear maps $f:\LT(n)\lto\LT(n)$ as follows, for all $Z,Z'\in\LT(n)$:
\begin{equation}
    Z\cdot Z'=\dotprod{f(Z)}{Z'}=\dotprod{Z}{f(Z')}.
\end{equation}
In other words, if $\vecto:\LT(n)\lto\R^{n(n+1)/2}$ denotes the linear map defined by $\vecto(Z)=(Z_{11},Z_{21},Z_{22},...,Z_{n,n-1},Z_{nn})^\top$ for all $Z\in\LT(n)$, then there exists an SPD matrix $A\in\Sym^+(\frac{n(n+1)}{2})$ such that $Z\cdot Z'=\vecto(Z)^\top A\,\vecto(Z')$. Thus the Lie-Cholesky metrics are parametrized by self-adjoint positive definite linear maps $f:\LT(n)\lto\LT(n)$ or SPD matrices $A\in\Sym^+(\frac{n(n+1)}{2})$. They are denoted $\mathrm{LC}(f)$ or $\mathrm{LC}(A)$.

Therefore, we can see the family of Lie-Cholesky metrics as a natural non-permu\-tation-invariant extension of affine-invariant metrics. Since they are left-invariant metrics on a Lie group, the numerical computation of the geodesics (exponential map, logarithm map) and the parallel transport is simpler and much more stable and precise than for other metrics \cite{Guigui21}. The curvature is known in closed form modulo the computation of the operator $\ad^*:\LT(n)\times\LT(n)\lto\LT(n)$ defined for all $X,Y,Z\in\LT(n)$ by $\dotprod{\ad^*(X)(Y)}{Z}=\dotprod{Y}{[X,Z]}$ \cite[Theorem 7.30]{Besse87}. The operator $\ad^*$ for Lie-Cholesky metrics is given in the following lemma.

\begin{lemma}[Operator $\ad^*$ for Lie-Cholesky metrics]
~\\Let $\mathrm{LC}(A)$ be a Lie-Cholesky metric characterized by $A\in\Sym^+(\frac{n(n+1)}{2})$. Then for $X,Y\in\LT(n)$, $\vecto(\ad^*(X)(Y))=A(I_n\otimes X^\top-X\otimes I_n)A^{-1}\vecto(Y)$.
\end{lemma}

\begin{proof}
In the matrix Lie group $\LT(n)$, the Lie bracket is the commutator $[X,Z]=XZ-ZX$. Therefore:
\begin{align*}
    \ad^*(X)(Y)\cdot Z&=Y\cdot[X,Z]=\vecto(Y)^\top A\,\vecto(XZ-ZX)\\
    &=\vecto(Y)^\top A(I_n\otimes X-X^\top\otimes I_n)\vecto(Z)\\
    &=[\vecto(Y)^\top A(I_n\otimes X-X^\top\otimes I_n)A^{-1}]A\,\vecto(Z)\\
    \vecto(\ad^*(X)(Y))&=A^{-1}(I_n\otimes X^\top-X\otimes I_n)A\,\vecto(Y).
\end{align*}
\end{proof}

\subsection{Quotient-Lie-Cholesky metrics}

Lie-Cholesky metrics being $\LT^+(n)$-invariant, they are in particular $\Diag^+(n)$-invariant as the affine-invariant metric. As recalled earlier, the smooth map $\cor:\Sym^+(n)\lto\Cor^+(n)$ is a $\Diag^+(n)$-invariant submersion so Lie-Cholesky metrics descend to quotient-Lie-Cholesky metrics on full-rank correlation matrices. Quotient-Lie-Cholesky metrics naturally extend quotient-affine metrics. To use them, one needs to compute the vertical and horizontal projections. After the definition of quotient-Lie-Cholesky metrics, the following lemma gives the expression of the vertical and horizontal distributions.

\begin{definition}[Quotient-Lie-Cholesky metric]
A quotient-Lie-Cholesky metric is the quotient metric of a Lie-Cholesky metric by the submersion $\cor:\Sym^+(n)\lto\Cor^+(n)$. Equivalently, it is the pushforward metric by the diffeomorphism $\Phi:L\in\LT^1(n)\lmto\Diag(LL^\top)^{-1/2}LL^\top\Diag(LL^\top)^{-1/2}\in\Cor^+(n)$ of a quotient metric on $\LT^1(n)\simeq\Diag(n)\backslash{\LT^+(n)}$ defined by quotient via the submersion $\pi:L\in\LT^+(n)\lmto\Diag(L)^{-1}L\in\LT^1(n)$.
\end{definition}

In this section, we express all the Riemannian operations of quotient-Lie-Cholesky metrics $\LT^1(n)$. It suffices to push them forward by $\Phi$ to get them on $\Cor^+(n)$.

\begin{lemma}[Vertical distribution, horizontal distribution]
The vertical and horizontal distributions associated to the quotient-Lie-Cholesky associated to $f$ are $\mc{V}_L=\Diag(n)L$ and $\mc{H}^{\mathrm{LC}(f)}_L=L\,f^{-1}(L^{-1}d_{LL^\top}\Chol(\mc{S}_{(LL^\top)^{-1}}(\Hol(n))))$.
\end{lemma}

\begin{proof}
The vertical space is the tangent space of the fiber $\Diag^+(n)L$ or the kernel of the differential of the submersion $d_L\pi(Z)=\Diag(L)^{-1}(Z-\Diag(L^{-1}Z)L)$. Hence $\mc{V}_L=\Diag(n)L$. Moreover, we know that when $f=\Id_{\LT(n)}$, the horizontal space is $\mc{H}^{\mathrm{AI}}_L=d_{LL^\top}\Chol(\mc{S}_{(LL^\top)^{-1}}(\Hol(n)))$. Indeed, it is the pushforward by the Cholesky map of the horizontal space on SPD matrices (cf. Section 2.2). In other words, for all $Z\in\mc{H}^{\mathrm{AI}}_L$, for all $Z'\in\mc{V}_L$, we have $\dotprod{L^{-1}Z}{L^{-1}Z'}=0$. Therefore, the horizontal space of the metric $\mathrm{LC}(f)$ is:
\begin{align*}
    \mc{H}^{\mathrm{LC}(f)}_L&=\{Z\in\LT(n)|\,\forall Z'\in\mc{V}_L,\dotprod{f(L^{-1}Z)}{L^{-1}Z'}=0\}\\
    &=\{Z\in\LT(n)|\,\forall Z'\in\mc{V}_L,\dotprod{L^{-1}Lf(L^{-1}Z)}{L^{-1}Z'}=0\}\\
    &=\{Z\in\LT(n)|\,Lf(L^{-1}Z)\in\mc{H}^{\mathrm{AI}}_L\}\\
    &=Lf^{-1}(L^{-1}\mc{H}^{\mathrm{AI}}_L).
\end{align*}
\end{proof}

Then, to compute the vertical and horizontal projections, it suffices to take bases of the vertical space $\mc{V}_L$ and the horizontal space $\mc{H}_L$, to orthonormalize them by Gram-Schmidt process and to project onto these orthonormal bases. Then the horizontal lift can be computed as follows, which allows to compute the metric and the exponential map.

\begin{lemma}[Horizontal lift, Riemannian metric, exponential map]
Given the horizontal projection $\hor_L:T_L\LT^+(n)\lto\mc{H}_L$ and the exponential map of the Lie-Cholesky metric $\Exp^{\mathrm{LC}}:\LT(n)\lto\LT^+(n)$, we have for all $L\in\LT^+(n)$, for all $\Gamma=\Diag(L)^{-1}L\in\LT^1(n)$, for all $\xi\in T_\Gamma\LT^1(n)\simeq\LT^0(n)$:
\begin{enumerate}[label=$\cdot$]
    \itemsep0em
    \item (Horizontal lift) $\xi^\#_L=\hor_L(\Diag(L)\xi)$,
    \item (Riemannian metric) $g^{\mathrm{QLC}}_\Gamma(\xi,\xi)=\dotprod{f(\Gamma^{-1}\hor_\Gamma(\xi))}{\Gamma^{-1}\hor_\Gamma(\xi)}$,
    \item (Exponential map) $\Exp^{\mathrm{QLC}}_\Gamma(t\xi)=\Diag(\Exp^{\mathrm{LC}}_\Gamma(t\,\hor(\xi)))^{-1}\Exp^{\mathrm{LC}}_\Gamma(t\,\hor(\xi))$.
\end{enumerate}
In particular, quotient-Lie-Cholesky metrics are geodesically complete.
\end{lemma}

\begin{proof}
For all $\xi\in\LT^0(n)$, we have $d_L\pi(\Diag(L)\xi)=\Diag(L)^{-1}(\Diag(L)\xi-\Diag(\xi)L)=\xi$. Hence $\xi^\#_L=\hor_L(\Diag(L)\xi)$. In particular, $\xi^\#_\Gamma=\hor_\Gamma(\xi)$. Then the metric and the exponential map simply write $g^{\mathrm{QLC}}_\Gamma(\xi,\xi)=g^{\mathrm{LC}}_\Gamma(\xi^\#_\Gamma,\xi^\#_\Gamma)$ and $\Exp^{\mathrm{QLC}}_\Gamma(t\xi)=\pi(\Exp^{\mathrm{LC}}_\Gamma(t\xi^\#_\Gamma)$.
\end{proof}

The Riemannian logarithm and the Riemannian distance can then be computed by minimizing the Lie-Cholesky distance along a fiber. The Lie-Cholesky distance is itself computed numerically with efficient tools on Lie groups \cite{Guigui21}.

In this section, we gave an overview of the congruence action of several matrix Lie groups on SPD matrices. It allowed to understand that the natural extension of affine-invariant metrics to non-permutation-invariant metrics is the family of $\LT^+(n)$-invariant metrics. We showed that such metrics are pullbacks of left-invariant metrics on the Lie group $\LT^+(n)$ by the Cholesky diffeomorphism. Hence, the space of lower triangular matrices and the Cholesky map naturally appear when one wants to get rid of the invariance under permutations on SPD matrices. Following the same construction as for the quotient-affine metric, we built quotient-Lie-Cholesky metrics and we showed that most of the interesting Riemannian operations can be computed numerically. However, this geometry is not completely satisfying since we have no formula in closed form and no obvious nice geometric properties except geodesic completeness. Moreover, the numerical computation of the Riemannian logarithm and the Riemannian distance are done with the generic methods in quotient manifolds which can be heavy and unstable. Furthermore, the quotient-affine metric has unbounded curvature as shown in Section 2 and it is difficult to say something about the curvature of general quotient-Lie-Cholesky metrics.

Hence in the following section, we continue to rely on lower triangular matrices and the Cholesky map to define Riemannian metrics with simpler geometries than the general quotient geometry of quotient-Lie-Cholesky metrics.

\section{New geometric structures with unique mean}

In this section, we define new families of metrics on the open elliptope of full-rank correlation matrices. They are based on simple geometries of subspaces of lower triangular matrices and transported to the elliptope via the Cholesky map or the derived map $\Theta$. In Section 4.1, we introduce the poly-hyperbolic-Cholesky metrics which provide Riemannian symmetric structures. Then we introduce the Euclidean-Cholesky metrics (Section 4.2) and the log-Euclidean-Cholesky metrics (Section 4.3) that provide Euclidean structures. Equipped with a metric of one of these three families, the open elliptope is Hadamard so the Riemannian logarithm and Fréchet mean are unique. In Section 4.4, we introduce a Lie group structure that allows to define Cartan-Schouten affine connections and left-invariant metrics. The group mean of the canonical Cartan-Schouten connection is unique. In Section 4.5, we give the geodesics in dimension 2, which correspond to interpolations of one correlation coefficient.

\subsection{Riemannian symmetric space structure: poly-hyperbolic-Choles\-ky metrics}

In this section, we use the diffeomorphism $\Chol_{|\Cor^+(n)}:\Cor^+(n)\lto\mc{L}$ where $\mc{L}$ is the set of lower triangular matrices with positive diagonal such that each row is unit normed for the canonical Euclidean norm. Thus, the $k$-th row of $L=\Chol(C)\in\mc{L}$ writes $(L_{k1},...,L_{k,k-1},L_{kk},0,...,0)$ with $L_{kk}>0$. It belongs to the open hemisphere $\mathrm{H}\Sph^{k-1}=\{x\in\R^k|\|x\|=1\mathrm{~and~}x_k>0\}$. So each $L\in\mc{L}$ is a point in $\mathrm{H}\Sph^0\times\cdots\times\mathrm{H}\Sph^{n-1}$. This construction is clearly bijective and diffeomorphic. Note that since $L_{11}=1$ and $\mathrm{H}\Sph^0=\{1\}$, we can remove it from the Cartesian product. Hence, we can define the diffeomorphism $\Psi:\mc{L}\lto\mathrm{H}\Sph^1\times\cdots\times\mathrm{H}\Sph^{n-1}$. Moreover, an open hemisphere is one of the avatars of the hyperbolic space, which is the Riemannian manifold of negative constant curvature.

Before introducing the Riemannian metric induced by the diffeomorphism $\Psi$, we recall the definition and Riemannian operations of the hyperbolic space in the model of the hyperboloid.

\begin{theorem}[Hyperbolic geometry of the hyperboloid]
The vector space $\R^{k+1}$ is endowed with the non-degenerate quadratic form $Q(x)=\sum_{i=1}^kx_i^2-x_{k+1}^2$. We denote $(x,y)\lmto Q(x,y)$ the associated symmetric bilinear form. The hyperboloid $\Hyp^k$ is the Riemannian manifold defined by $\Hyp^k=\{x\in\R^{k+1}|Q(x)=-1\}$ endowed with the induced pseudo-metric, which is a Riemannian metric on $\Hyp^k$. The tangent space writes $T_x\Hyp^k=\{v\in\R^{k+1}|Q(x,v)=0\}$. For all $x,y\in\Hyp^k$, $v\in T_x\Hyp^k$ and $w\in T_x\Hyp^k$ non colinear to $v$, the Riemannian operations are:
\begin{enumerate}[label=$\cdot$]
    \itemsep0em
    \item (Riemannian metric) $g^{\Hyp}_x(v,v)=Q(v)$,
    \item (Riemannian norm) $\|v\|_\Hyp=\sqrt{Q(v)}$,
    \item (Riemannian distance) $d_\Hyp(x,y)=\arccos(-Q(x,y))$,
    \item (Exponential map) $\Exp_x^\Hyp(v)=\cosh(\|v\|_\Hyp)x+\sinh(\|v\|_\Hyp)\frac{v}{\|v\|_\Hyp}$,
    \item (Logarithm map) $\Log^\Hyp_x(y)=d^\Hyp(x,y)\frac{y+Q(x,y)x}{\|y+Q(x,y)x\|_\Hyp}$,
    \item (Sectional curvature) $\kappa^\Hyp_x(v,w)=-1$.
\end{enumerate}

The formulae on the other models of the hyperbolic space can be obtained by pullback via the appropriate diffeomorphism. In particular, the diffeomorphism between the open hemisphere $\mathrm{H}\Sph^k=\{(x_1,...,x_{k+1})\in\R^k\times\R^+|\sum_{i=1}^{k+1}x_i^2=1\}$ and the hyperbolic space $\Hyp^k$ is:
\begin{equation}
    \varphi^{\Sph\Hyp}:(x_1,...,x_{k+1})\in\mathrm{H}\Sph^k\lmto\frac{1}{x_{k+1}}(x_1,...,x_k,1)\in\Hyp^k.
\end{equation}
Thus the natural metric on the open hemisphere is the pullback metric $g^{\mathrm{H}\Sph}=(\varphi^{\Sph\Hyp})^*g^{\Hyp}$.
\end{theorem}

\begin{definition}[Poly-hyperbolic-Cholesky metrics]
Let $\alpha_1,...,\alpha_{n-1}>0$ be positive coefficients. A poly-hyperbolic-Cholesky metric on $\Cor^+(n)$ is the pullback metric $g^{\mathrm{PHC}}=(\Psi\circ\Chol)^*(\alpha_1g^{\mathrm{H}\Sph^1}\oplus...\oplus \alpha_{n-1}g^{\mathrm{H}\Sph^{n-1}})$ by the map $\Chol\circ\Psi$ of a weighted product metric on the product of hyperbolic spaces $\mathrm{H}\Sph^1\times\cdots\times\mathrm{H}\Sph^{n-1}$. The PHC metric with all weights equal to 1 is called the canonical PHC metric.
\end{definition}

\begin{theorem}[Symmetric space structure]
The manifold of full-rank correlation matrices $\Cor^+(n)$ equipped with a poly-hyperbolic-Cholesky metric is a Riemannian symmetric space of non-positive sectional curvature bounded by $[a,0]$ with $a=-\frac{1}{\min_{i\gs 2}\alpha_i}$. For $n\gs 3$, it is not of constant curvature. The canonical PHC metric writes for all $C\in\Cor^+(n)$ and $X\in T_C\Cor^+(n)\simeq\Hol(n)$:
\begin{equation}
    g^{\mathrm{CPHC}}_C(X,X)=\|\Diag(L)^{-1}L\,\low_\S(L^{-1}XL^{-\top})\|^2,
\end{equation}
where $L=\Chol(C)\in\mc{L}$. The square distance between $C$ and $C'=\phi(L')$ writes:
\begin{equation}
    d_{\mathrm{CPHC}}(C,C')^2=\sum_{i=2}^n\arccos(-Q(L_i^\top,{L_i'}^\top))^2,
\end{equation}
where $L_i,L_i'$ are the $i$-th rows of $L,L'$ respectively.
\end{theorem}

\begin{proof}
The product of Riemannian symmetric spaces is a Riemannian symmetric space. The product of manifolds with sectional curvature bounded by $[a,b]$ with $a\ls 0\ls b$ has its sectional curvature bounded by $[a,b]$. (This is also valid for $(-\infty,b]$ and $[a,+\infty)$.) Let $k\gs 2$ such that $\alpha_k=\min_{i\gs 2}\alpha_i$. The values $a=-\frac{1}{\alpha_k}$ and $b=0$ are clearly reached for $n\gs 3$ by bivectors $(X,Y)$ and $(X,Z)$ respectively, where $X=(0,...,0,X_k,0,...,0)$, $Y=(0,...,0,Y_k,0,...,0)$ and $Z=(Z_1,0,...,0)\in T(\mathrm{H}\Sph^1\times\cdots\times\mathrm{H}\Sph^{n-1})$. To express the Riemannian metric, we take the pullback of the Riemannian metric on the hyperboloid $\Hyp^n$ defined by $g_x(v,v)=\sum_{k=1}^n{v_k^2}-v_{n+1}^2$ by the diffeomorphism $\varphi^{\mathbb{SH}}:(x_1,...,x_{n+1})\in\mathrm{H}\Sph^n\lmto\frac{1}{x_{n+1}}(x_1,...,x_n,1)\in\Hyp^n$. We compute for all $x\in\mathrm{H}\Sph^n$ and all $v\in T_x\mathrm{H}\Sph^n$, using $\sum_{k=1}^{n+1}x_k^2=1$ and $\sum_{k=1}^{n+1}x_kv_k=0$:
\begin{align*}
    d_x\varphi^{\mathbb{SH}}(v)&=\frac{1}{x_{n+1}}\left(v_1-x_1\frac{v_{n+1}}{x_{n+1}},...,v_n-x_n\frac{v_{n+1}}{x_{n+1}},-\frac{v_{n+1}}{x_{n+1}}\right),\\
    g^{\mathrm{H}\Sph}_x(v,v)&=g^{\Hyp}_{\varphi(x)}(d_x\varphi^{\mathbb{SH}}(v),d_x\varphi^{\mathbb{SH}}(v))\\
    &=\frac{1}{x_{n+1}^2}\left(\sum_{k=1}^n\left(v_k-x_k\frac{v_{n+1}}{x_{n+1}}\right)^2-\frac{v_{n+1}^2}{x_{n+1}^2}\right)\\
    &=\frac{1}{x_{n+1}^2}\left(\sum_{k=1}^n\left(v_k^2-2x_kv_k\frac{v_{n+1}}{x_{n+1}}+x_k^2\frac{v_{n+1}^2}{x_{n+1}^2}\right)-\frac{v_{n+1}^2}{x_{n+1}^2}\right)\\
    &=\frac{1}{x_{n+1}^2}\left(\sum_{k=1}^nv_k^2+2v_{n+1}^2+(1-x_{n+1}^2)\frac{v_{n+1}^2}{x_{n+1}^2}-\frac{v_{n+1}^2}{x_{n+1}^2}\right)\\
    &=\frac{\|v\|^2}{x_{n+1}^2}.
\end{align*}
Hence, for all $C\in\Cor^+(n)$, $X\in\Hol(n)$, $L=\Chol(C)\in\mc{L}$ and $Y=d_C\Chol(X)=L\,\low_\S(L^{-1}XL^{-\top})$ with $Y_{11}=\frac{1}{2}L_{11}[L^{-1}XL^{-\top}]_{11}=\frac{X_{11}}{2L_{11}}=0$, we have:
\begin{align*}
    g^{\mathrm{PHC}}_C(X,X)&=g^{\mathrm{H}\Sph^1\times\cdots\times\mathrm{H}\Sph^{n-1}}_L(Y,Y)\\
    &=\sum_{i=2}^n\frac{\|Y_{i\bullet}\|^2}{L_{ii}^2}=\sum_{i=1}^n\|\Diag(L)^{-1}_{ii}Y_{i\bullet}\|^2\\
    &=\|\Diag(L)^{-1}Y\|^2=\|\Diag(L)^{-1}L\,\low_\S(L^{-1}XL^{-\top})\|^2.
\end{align*}

Note that the general PHC metric writes $g^{\mathrm{PHC}}_C(X,X)=\sum_{i=2}^n\|\Diag(L)_{ii}^{-1}\alpha_iY_{i\bullet}\|^2$ and the general PHC distance writes $d(C,C')^2=\sum_{i=2}^n\alpha_i\arccos(-Q(L_i^\top,{L_i'}^\top))^2$.
\end{proof}

\subsection{Vector space structure: Euclidean-Cholesky metrics}

The Cholesky map was already used on SPD matrices \cite{Wang04}. The Euclidean-Cholesky metric is defined as the Euclidean metric on the Cholesky factor belonging to $\LT^+(n)$. However, it is not complete because $\LT^+(n)$ is open in $\LT(n)$. A solution to this problem is to take the logarithm of the diagonal before taking the Euclidean metric \cite{Pinheiro96,Lin19}. It amounts to define the product metric of a Euclidean metric on the strictly lower part and a log-Euclidean metric on the diagonal part.

For correlation matrices, we can use the diffeomorphism $\Theta:\Cor^+(n)\lto\LT^1(n)$. Since the diagonal is $I_n$, the two mentioned metrics reduce to the same metric on the open elliptope. We call it the Euclidean-Cholesky metric.

\begin{definition}[Euclidean-Cholesky metrics]
The Euclidean-Cholesky metrics on full-rank correlation matrices are pullback metrics by $\Theta:\Cor^+(n)\lto\LT^1(n)$ of inner products on $\LT^1(n)=I_n+\LT^0(n)$.
\end{definition}

\begin{theorem}[Riemannian operations]
Let $\|\cdot\|$ be a Euclidean norm on $\LT^0(n)$. The Riemannian operations of the Euclidean-Cholesky metric associated to this norm are, for all $C,C',C_i\in\Cor^+(n)$, $X\in T_C\Cor^+(n)\simeq\Hol(n)$, $t\in\R$:
\begin{enumerate}[label=$\cdot$]
    \itemsep0em
    \item (Exponential map) $\Exp_C(tX)=\Theta^{-1}(\Theta(C)+t\,d_C\Theta(X))$,
    \item (Logarithm map) $\Log_C(C')=(d_C\Theta)^{-1}(\Theta(C')-\Theta(C))$,
    \item (Geodesic) $\gamma_{C\to C'}(t)=\Theta^{-1}((1-t)\Theta(C)+t\,\Theta(C'))$,
    \item (Distance) $d(C,C')=\|\Theta(C')-\Theta(C)\|$,
    \item (Parallel transport) $\Pi_{C\to C'}X=(d_{C'}\Theta)^{-1}(d_C\Theta(X))$,
    \item (Curvature) Null,
    \item (Fréchet mean) $\bar{C}=\Theta^{-1}(\frac{1}{n}\sum_{i=1}^n\Theta(C_i))$,
\end{enumerate}
where $d_C\Theta(X)=\Theta(C)\low_\S(L^{-1}XL^{-\top})-\frac{1}{2}\Diag(L^{-1}XL^{-\top})\Theta(C)$ and $L=\Chol(C)$. The Euclidean-Cholesky metrics are geodesically complete.
\end{theorem}

These metrics are flat, geodesically complete and the Riemannian operations are trivial. Since they reduce to (the pullback of) an inner product on a \textit{vector space}, we prefer not to use the term \textit{Lie group} for them, contrarily to the terminology of \cite{Li17,Lin19}. We prefer to reserve it for Lie groups that are not vector spaces, such as the natural Lie group structure of $\LT^+(n)$ (with matrix multiplication) underlying Lie-Cholesky metrics.

\subsection{Vector space structure: log-Euclidean-Cholesky metrics}

Another map was used to Euclideanize the manifold $\LT^+(n)$: the matrix logarithm \cite{Li17}. Indeed, the matrix exponential is a smooth diffeomorphism from $\LT(n)$ to $\LT^+(n)$ \cite{Gallier08}. We can use the same idea for correlation matrices since the matrix exponential is a smooth diffeomorphism from $\LT^0(n)$ to $\LT^1(n)$. Moreover it has a particularly simple expression:
\begin{equation}
    \exp(\xi)=\sum_{k=0}^{n-1}\frac{1}{k!}\xi^k,
\end{equation}
because $\LT^0(n)$ is a nilpotent algebra. Then the logarithm $\log:\LT^1(n)\lto\LT^0(n)$ is simply:
\begin{equation}
    \log(Z)=\sum_{k=1}^{n-1}\frac{(-1)^{k-1}}{k}(Z-I_n)^k.
\end{equation}
Therefore, the differential of the logarithm writes:
\begin{equation}
    d_Z\log(\xi)=\sum_{k=1}^{n-1}\frac{(-1)^k}{k}[(Z-I_n)^{k-1}\xi+(Z-I_n)^{k-2}\xi(Z-I_n)+...+\xi(Z-I_n)^{k-1}].
\end{equation}

\begin{definition}[Log-Euclidean-Cholesky metrics]
The log-Euclidean-Cholesky metrics on full-rank correlation matrices are pullback metrics by $\log\circ\,\Theta:\Cor^+(n)\lto\LT^0(n)$ of inner products on $\LT^0(n)$.
\end{definition}

\begin{theorem}[Riemannian operations]
Let $\|\cdot\|$ be a Euclidean norm on $\LT^0(n)$. The Riemannian operations of log-Euclidean-Cholesky metrics associated to this norm are, for all $C,C',C_i\in\Cor^+(n)$, $X\in T_C\Cor^+(n)\simeq\Hol(n)$, $t\in\R$:
\begin{enumerate}[label=$\cdot$]
    \itemsep0em
    \item (Exponential map) $\Exp_C(tX)=\Theta^{-1}\circ\exp(\log(\Theta(C))+t\,d_C(\log\circ\,\Theta)(X))$,
    \item (Logarithm map) $\Log_C(C')=(d_C(\log\circ\,\Theta))^{-1}(\log(\Theta(C'))-\log(\Theta(C)))$,
    \item (Geodesic) $\gamma_{C\to C'}(t)=\Theta^{-1}\circ\exp((1-t)\log(\Theta(C))+t\,\log(\Theta(C')))$,
    \item (Distance) $d(C,C')=\|\log(\Theta(C'))-\log(\Theta(C))\|$,
    \item (Parallel transport) $\Pi_{C\to C'}X=(d_{C'}(\log\circ\,\Theta))^{-1}(d_C\log\circ\,\Theta(X))$,
    \item (Curvature) Null,
    \item (Fréchet mean) $\bar{C}=\Theta^{-1}\circ\exp(\frac{1}{n}\sum_{i=1}^n\log(\Theta(C_i)))$,
\end{enumerate}
where $d_C(\log\circ\,\Theta)(X)=d_{\Theta(C)}\log(d_C\Theta(X))$. The log-Euclidean-Cholesky metrics are geodesically complete.
\end{theorem}

The computation of the Riemannian operations of the Euclidean-Cholesky metrics is more straightforward than those of the log-Euclidean-Cholesky because one has to compute the differential of the triangular matrix logarithm for the latter.

\subsection{Nilpotent Lie group structure}

Another interesting structure is given by the natural Lie group structure of $\LT^1(n)$ for the matrix multiplication. This equips full-rank correlation matrices with a Lie group structure via the diffeomorphism $\Theta:\Cor^+(n)\lto\LT^1(n)$. Hence, left-invariant metrics can be defined. In analogy with $\Sym^+(n)\simeq\LT^+(n)$, they can also be called Lie-Cholesky metrics. Then all Riemannian operations can be computed numerically \cite{Guigui21} and the space is ensured to be geodesically complete. However, this doesn't give information on the sign of the curvature.

More interestingly, one can rely on the canonical Cartan-Schouten connection to define the group exponential and the notion of group mean. We can also name them after Lie-Cholesky.

\begin{theorem}[Group operations]
The group operations associated to the Lie-Cholesky group structure on full-rank correlation matrices are, for all $C,C',C_i\in\Cor^+(n)$, $X\in T_C\Cor^+(n)\simeq\Hol(n)$, $t\in\R$:
\begin{enumerate}[label=$\cdot$]
    \itemsep0em
    \item (Exponential map) $\Exp_C(tX)=\Theta^{-1}(\Theta(C)\exp(t\,\Theta(C)^{-1}d_C\Theta(X))$,
    \item (Logarithm map) $\Log_C(C')=(d_C\Theta)^{-1}(\Theta(C)\log(\Theta(C)^{-1}\Theta(C')))$,
    \item (Geodesic) $\gamma_{C\to C'}(t)=\Theta^{-1}(\Theta(C)(\Theta(C)^{-1}\Theta(C'))^t)$,
    \item (Group mean) Unique, characterized by $\sum_{i=1}^k\log(\Theta(\bar{C})^{-1}\Theta(C_i))=0$.
\end{enumerate}
\end{theorem}

\begin{proof}
The exponential map, logarithm map and geodesics are pullbacks by $\Theta$ of corresponding operations in $\LT^1(n)$, which are for all $\Gamma,\Gamma'\in\LT^+(n)$, $\xi\in T_\Gamma\LT^+(n)\simeq\LT(n)$ and $t\in\R$:
\begin{enumerate}[label=$\cdot$]
    \itemsep0em
    \item (Exponential map) $\Exp_\Gamma(t\xi)=\Gamma\exp(t\,\Gamma^{-1}\xi)$,
    \item (Logarithm map) $\Log_\Gamma(\Gamma')=\Gamma\log(\Gamma^{-1}\Gamma')$,
    \item (Geodesic) $\gamma_{\Gamma\to \Gamma'}(t)=\Exp_\Gamma(t\,\Gamma^{-1}\Log_\Gamma(\Gamma'))=\Gamma(\Gamma^{-1}\Gamma')^t$.
\end{enumerate}
Since the Lie algebra $\LT^0(n)$ is nilpotent, the group mean $\bar{\Gamma}$ of the finite sample $\Gamma_1,...,\Gamma_k\in\LT^1(n)$ is unique \cite[Example 8.1.8]{Buser81}. It is characterized by $0=\sum_{i=1}^k\Log_{\bar{\Gamma}}(\Gamma_i)=\bar{\Gamma}\sum_{i=1}^k\log(\bar{\Gamma}^{-1}\Gamma_i)$, which is equivalent to $\sum_{i=1}^k\log(\bar{\Gamma}^{-1}\Gamma_i)=0$.
\end{proof}

\subsection{Explicit geodesics in dimension 2}

In dimension 2, the elliptope is reduced to one parameter. All full-rank correlation matrices write $C=C(\rho)=\begin{pmatrix}1 & \rho\\\rho & 1\end{pmatrix}$ with $\rho\in(-1,1)$. Therefore, the quotient-affine metric and the metrics defined in Section 4 only depend on one scaling parameter. They actually split in two groups and the geodesics can be computed in closed forms. The two formulae in the following result provide two different interpolations of the correlation coefficient. The proof is in the supplementary material.

\begin{theorem}[Geodesics in dimension 2]\label{thm:dimension2}
Let $C_1=C(\rho_1),C_2=C(\rho_2)\in\Cor^+(2)$ with $\rho_1,\rho_2\in(-1,1)$.
\begin{enumerate}
    \item Quotient-affine metrics and poly-hyperbolic-Cholesky metrics coincide (up to a scaling factor). The geodesic between $C_1$ and $C_2$ is $C(\rho(t))$ for $t\in\R$ where:
    \begin{equation}
        \rho(t)=\frac{\rho_1\cosh(\lambda t)+\sinh(\lambda t)}{\rho_1\sinh(\lambda t)+\cosh(\lambda t)},
    \end{equation}
    where $\lambda=\log\sqrt\frac{1+\rho_2}{1-\rho_2}-\log\sqrt\frac{1+\rho_1}{1-\rho_1}$ is known as the difference of the Fisher transformation of the correlation coefficients $\rho_1$ and $\rho_2$.
    \item Euclidean-Cholesky and log-Euclidean-Cholesky metrics coincide. The geodesic between $C_1$ and $C_2$ is $C(\rho(t))$ for $t\in\R$ where:
    \begin{equation}
        \rho(t)=\frac{F(t)}{\sqrt{1+F(t)^2}},
    \end{equation}
    where $F(t)=(1-t)\frac{\rho_1}{\sqrt{1-\rho_1^2}}+t\frac{\rho_2}{\sqrt{1-\rho_2^2}}$. This geodesic also coincides with the Lie-Cholesky group geodesic of Section 4.4.
\end{enumerate}
\end{theorem}

\section{Conclusion}

In this work, we proposed new Riemannian metrics on the open elliptope of full-rank correlation matrices that are not invariant under permutations. To the best of our knowledge, all the existing geometric structures were invariant under permutations. Thus the geometries we propose significantly departs from the classical structures. This can be a good assumption in some applications and an irrelevant characteristic in some others. We generalized the recently introduced quotient-affine metrics by studying the congruence action of several matrix Lie groups on SPD matrices. We showed that the family of $\LT^+(n)$-invariant metrics is a natural non-permutation-invariant generalization of affine-invariant metrics. Moreover, they are pullbacks of left-invariant metrics on the Lie group $\LT^+(n)$ by the Cholesky map. They are invariant under the congruence action of positive diagonal matrices so they descend to the elliptope. We explained that the main Riemannian operations can be computed numerically for these quotient-Lie-Cholesky metrics. However, we also showed that the curvature of quotient-affine metrics is unbounded and we can conjecture that the situation is not better for quotient-Lie-Cholesky metrics. In addition, the Riemannian operations are not computable in closed form a priori.

That is why we introduced new Riemannian metrics on the elliptope in a different way. We kept the Cholesky map which seems to be a good alternative to the invariance under permutations since the space of lower triangular matrices is not stable by permutations. Thus we defined the poly-hyperbolic-Cholesky (PHC) metrics which provide non-positively curved Riemannian symmetric space structures. We also defined two kinds of vector space structures that are flat, geodesically complete and for which all operations are known in closed form. Thus, these three families of metrics provide a Hadamard structure, in particular the Riemannian logarithm and the Fréchet mean are unique. We also put forward a nilpotent Lie group structure for which the group mean is unique. Finally, we proved that in dimension 2, the PHC geodesics are the quotient-affine geodesics and the geodesics of the three last structures coincide. This provides a new interpolation of the correlation coefficient.

It would be nice to test these new metrics on different kinds of data in future works. Moreover, all metrics on correlation matrices provide new product metrics on covariance matrices by decoupling the scales of the variables and the correlations between them. This approach seems promising since in many problems, the correlation gives more information than the covariance on the strength of the relations between the variables, although the scales can remain interesting. Thus, one question could be to adjust the weights between the two components and also between the scales of the variables. The possibilities are multiplied now we have many metrics on correlation matrices. Another direction of research is to investigate permutation-invariant Riemannian metrics on correlation matrices with a simpler geometry than the one of the quotient-affine metrics, for example Hadamard or even flat.

\section*{Acknowledgments}
This project has received funding from the European Research Council (ERC) under the European Union’s Horizon 2020 research and innovation program (grant G-Statistics agreement No 786854). This work has been supported by the French government, through the UCAJEDI Investments in the Future project managed by the National Research Agency (ANR) with the reference number ANR-15-IDEX-01 and through the 3IA Côte d’Azur Investments in the Future project managed by the National Research Agency (ANR) with the reference number ANR-19-P3IA-0002. The authors warmly thank Nicolas Guigui for insightful discussions on lower triangular matrices.

\begin{appendices}

\section{Proof appendix of Theorem 2.1}

Let $X=\mathds{11}^\top-I_n$ and $Y=\mu\mathds{1}^\top+\mathds{1}\mu^\top-2\,\diag(\mu)$ with $\somme(\mu)=\mathds{1}^\top\mu=0$ where $\mu\in\R^n$. Let $C=(1-\rho)I_n+\rho\mathds{11}^\top\in\Cor^+(n)$ for $\rho\in(-\frac{1}{n-1},1)$. Let us show that $\kappa_C(X,Y)$ tends to $+\infty$ when $\rho\to-\frac{1}{n-1}$, which proves that the curvature is not bounded from above.

The symmetric matric $\mathds{11}^\top$ has two eigenvalues: $0$ with multiplicity $n-1$ and $n$ with multiplicity $1$. Since $(\mathds{11}^\top)^2=n\mathds{11}^\top$, the minimal polynomial is $\mathrm{P}_{\mathds{11}^\top}(x)=x(x-n)$ for $x\in\R$. For all $\alpha,\beta\in\R$, the symmetric matrix $\Sigma=\alpha I_n+\beta\mathds{11}^\top$ has minimal polynomial $\mathrm{P}_\Sigma(x)=(x-\alpha)(x-(\alpha+n\beta))$ for $x\in\R$, which is of degree 2. Hence $\Sigma$ is positive definite if and only if $\alpha>0$ and $\alpha+n\beta>0$. In this case, its inverse is a polynomial in $\mathds{11}^\top$ of degree 1. More precisely, $\Sigma^{-1}=\alpha' I_n+\beta'\mathds{11}^\top$ with $\alpha'=\frac{1}{\alpha}$ and $\alpha'+n\beta'=\frac{1}{\alpha+n\beta}$, i.e. $\beta'=-\frac{\beta}{\alpha(\alpha+n\beta)}$. Note that $\alpha\beta'+\beta\alpha'+n\beta\beta'=0$.

Moreover, for all $i\ne j\in\{1,...,n\}$, $[\Sigma\bullet\Sigma^{-1}]_{ii}=(\alpha+\beta)(\alpha'+\beta')$ and $[\Sigma\bullet\Sigma^{-1}]_{ij}=\beta\beta'$. Therefore, $I_n+\Sigma\bullet\Sigma^{-1}=AI_n+B\mathds{11}^\top$ with $A=1+(\alpha+\beta)(\alpha'+\beta')-\beta\beta'=2+\alpha\beta'+\beta\alpha'=\frac{2\alpha(\alpha+n\beta)+n\beta^2}{\alpha(\alpha+n\beta)}$ and $B=\beta\beta'$. Note that $A+nB=2$. And $(I_n+\Sigma\bullet\Sigma^{-1})^{-1}=A'I_n+B'\mathds{11}^\top$ with $A'=\frac{1}{A}$ and $B'=-\frac{B}{2A}$.

We compute $\kappa_C(X,Y)$ where $C=\alpha I_n+\beta\mathds{11}^\top\in\Cor^+(n)$ with $\alpha+\beta=1$.
\small
\begin{align*}
	C^{-1}X&=(\alpha'I_n+\beta'\mathds{11}^\top)(\mathds{11}^\top-I_n)\\
	&=-\alpha'I_n+(\alpha'+(n-1)\beta')\mathds{11}^\top
\end{align*}
\begin{align*}
    &=-\frac{1}{\alpha}I_n+\frac{1}{\alpha(\alpha+n\beta)}\mathds{11}^\top,\\
	\Diag(C^{-1}X)&=(n-1)\beta'I_n,\\
	(I_n+C\bullet C^{-1})^{-1}\Diag(C^{-1}X)\mathds{1}&=(n-1)\beta'(A'I_n+B'\mathds{11}^\top)\mathds{1}\\
	&=\frac{n-1}{2}\beta'\mathds{1},\\
	X^\#&=X-(n-1)\beta'C,\\
	\Diag(X^\#)&=-(n-1)\beta'I_n,\\
	D(X,Y)&=\Diag(C^{-1}\Diag(X^\#)Y^\#-C^{-1}Y^\#C^{-1}\Diag(X^\#)C)\nonumber\\
	&=0,\\
	C^{-1}Y&=(\alpha'I_n+\beta'\mathds{11}^\top)(\mu\mathds{1}^\top+\mathds{1}\mu^\top-2\diag(\mu))\\
	&=\alpha'\mu\mathds{1}^\top+(\alpha'+(n-2)\beta')\mathds{1}\mu^\top-2\alpha'\diag(\mu),\\
	\Diag(C^{-1}Y)&=(n-2)\beta'\diag(\mu),\\
	(I_n+C\bullet C^{-1})^{-1}\Diag(C^{-1}Y)\mathds{1}&=(n-2)\beta'(A'I_n+B'\mathds{11}^\top)\mu\\
	&=(n-2)\beta'A'\mu,\\
	Y^\#&=Y-(n-2)\beta'A'(\diag(\mu)C+C\diag(\mu)),\\
    \Diag(Y^\#)&=-2(n-2)\beta'A'\diag(\mu),\\
    C^{-1}\Diag(Y^\#)X^\#&=-2(n-2)\beta'A'(\alpha'I_n+\beta'\mathds{11}^\top)\diag(\mu)(\mathds{11}^\top-I_n)\\
	&=-2(n-2)\beta'A'(-\alpha'\diag(\mu)+\alpha'\mu\mathds{1}^\top-\beta'\mathds{1}\mu^\top),
\end{align*}
\begin{align*}
    &C^{-1}X^\#C^{-1}\Diag(Y^\#)C\\
    &=-2(n-2)\beta'A'(-\alpha'I_n+(\alpha'+(n-1)\beta')\mathds{11}^\top)(\alpha'I_n+\beta'\mathds{11}^\top)\diag(\mu)(\alpha I_n+\beta\mathds{11}^\top) \\
    &=-2(n-2)\beta'A'(-\alpha'I_n+(\alpha'+(n-1)\beta')\mathds{11}^\top)(\diag(\mu)+\alpha'\beta\mu\mathds{1}^\top+\alpha\beta'\mathds{1}\mu^\top) \\
    &=-2(n-2)\beta'A'[-\alpha'(\diag(\mu)+\alpha'\beta\mu\mathds{1}^\top+\alpha\beta'\mathds{1}\mu^\top)+(\alpha'+(n-1)\beta')(1+n\alpha\beta')\mathds{1}\mu^\top],\\
    ~\\
    &C^{-1}\Diag(Y^\#)X^\#-C^{-1}X^\#C^{-1}\Diag(Y^\#)C \nonumber\\
    &=-2(n-2)\beta'A'[\alpha'(1+\alpha'\beta)\mu\mathds{1}^\top-(\alpha'+(n-1)\beta')(1+n\alpha\beta')\mathds{1}\mu^\top],
\end{align*}
\begin{align}
    D(Y,X)&=-2(n-2)\beta'A'\left(\frac{1}{\alpha^2}-\frac{1+(n-1)\alpha\beta'}{\alpha+n\beta}\right)\diag(\mu) \label{eq:alpha_plus_beta}\\
    &=-2(n-2)\beta'A'\left(\frac{1}{\alpha^2}-\frac{1}{(\alpha+n\beta)^2}\right)\diag(\mu) \nonumber\\
    &=2n(n-2)\frac{\beta^2}{2\alpha(\alpha+n\beta)+n\beta^2}\frac{2\alpha+n\beta}{\alpha^2(\alpha+n\beta)^2}\diag(\mu),\nonumber
\end{align}
\begin{align}
    \mathds{1}^\top D(I_n+C\bullet C^{-1})^{-1}D\mathds{1}&=A'\left(2n(n-2)\frac{\beta^2}{2\alpha(\alpha+n\beta)+n\beta^2}\frac{2\alpha+n\beta}{\alpha^2(\alpha+n\beta)^2}\right)^2\|\mu\|^2 \nonumber\\
    &=4n^2(n-2)^2\frac{\beta^4(2\alpha+n\beta)^2}{[\alpha(\alpha+n\beta)(2\alpha(\alpha+n\beta)+n\beta^2)]^3}\|\mu\|^2, \nonumber
\end{align}
\normalsize
where we used $\alpha+\beta=1$ from Equation (\ref{eq:alpha_plus_beta}).

Now, we compute $g^\QA_C(X,X)$, $g^\QA_C(Y,Y)$ and $g^\QA_C(X,Y)$.
\small
\begin{align*}
    g_C(X,X)&=\tr((C^{-1}X)^2)-2\mathds{1}^\top\Diag(C^{-1}X)(I_n+C\bullet C^{-1})^{-1}\Diag(C^{-1}X)\mathds{1}\\
    &=\tr\left((-\alpha'I_n+(\alpha'+(n-1)\beta')\mathds{11}^\top)^2\right)-2(n-1)^2{\beta'}^2\somme(A'I_n+B'\mathds{11}^\top)\\
    &=\tr({\alpha'}^2I_n+(n(\alpha'+(n-1)\beta')^2-2\alpha'(\alpha'+(n-1)\beta'))\mathds{11}^\top)\\
    &\quad-2(n-1)^2{\beta'}^2n(A'+nB')\\
    &=n((n-1){\alpha'}^2+2(n-1)^2\alpha'\beta'+n(n-1)^2{\beta'}^2)-n(n-1)^2{\beta'}^2\\
    &=\frac{n(n-1)}{\alpha^2(\alpha+n\beta)^2}((\alpha+n\beta)^2-2(n-1)\beta(\alpha+n\beta)+(n-1)^2\beta^2)\\
    &=\frac{n(n-1)}{\alpha^2(\alpha+n\beta)^2},\\
    g_C(Y,Y)&=\tr((C^{-1}Y)^2)-2\mathds{1}^\top\Diag(C^{-1}Y)(I_n+C\bullet C^{-1})^{-1}\Diag(C^{-1}Y)\mathds{1}\\
    &=\tr\left((\alpha'\mu\mathds{1}^\top+(\alpha'+(n-2)\beta')\mathds{1}\mu^\top-2\alpha'\diag(\mu))^2\right)\\
    &\quad-2(n-2)^2{\beta'}^2\mu^\top(A'I_n+B'\mathds{11}^\top)\mu\\
    &=\tr(4{\alpha'}^2\diag(\mu)^2+\alpha'(\alpha'+(n-2)\beta')(n\mu\mu^\top+\|\mu\|^2\mathds{11}^\top)) \nonumber\\
    &\quad -2\alpha'\tr((2\alpha'+(n-2)\beta')\mu\mu^\top+\alpha'(\mu\bullet\mu)\mathds{1}^\top+(\alpha'+(n-2)\beta')\mathds{1}(\mu\bullet\mu)^\top))\\
    &\quad-2(n-2)^2{\beta'}^2A'\|\mu\|^2
\end{align*}
\begin{align*}
    &=\|\mu\|^2(4{\alpha'}^2+2n\alpha'(\alpha'+(n-2)\beta')-4\alpha'(2\alpha'+(n-2)\beta')-2(n-2)^2{\beta'}^2A')\\
    &=\|\mu\|^2(2(n-2)\underset{\frac{1+\beta}{\alpha^2(\alpha+n\beta)}}{\underbrace{\alpha'(\alpha'+(n-2)\beta')}}-2(n-2)^2\underset{\frac{\beta^2}{\alpha(\alpha+n\beta)(2\alpha(\alpha+n\beta)+n\beta^2)}}{\underbrace{{\beta'}^2A'}})\\
    &=\frac{2(n-2)\|\mu\|^2}{\alpha^2(\alpha+n\beta)(2\alpha(\alpha+n\beta)+n\beta^2)}\underset{2\alpha(1+\beta)(\alpha+n\beta)+2\beta^2(\alpha+n\beta)=2(\alpha+n\beta)}{\underbrace{((1+\beta)(2\alpha(\alpha+n\beta)+n\beta^2)-(n-2)\alpha\beta^2)}}\\
    &=\frac{4(n-2)\|\mu\|^2}{\alpha^2(2\alpha(\alpha+n\beta)+n\beta^2)},
\end{align*}
\begin{align*}
    &g_C(X,Y)\\
    &=\tr(C^{-1}XC^{-1}Y)-2\mathds{1}^\top\Diag(C^{-1}X)(I_n+C\bullet C^{-1})^{-1}\Diag(C^{-1}Y)\mathds{1}\\
    &=\tr((-\alpha'I_n+(\alpha'+(n-1)\beta')\mathds{11}^\top)(\alpha'\mu\mathds{1}^\top+(\alpha'+(n-2)\beta')\mathds{1}\mu^\top-2\alpha'\diag(\mu)))\\
    &\quad -2(n-1)(n-2){\beta'}^2\mathds{1}^\top(A'I_n+\beta'\mathds{11}^\top)\mu\\
    &=\mathrm{constant}\times\somme(\mu)=0.
\end{align*}
\normalsize
Finally:
\begin{equation*}
    \frac{3}{8}\frac{\mathds{1}^\top D(I_n+C\bullet C^{-1})^{-1}D\mathds{1}}{g_C(X,X)g_C(Y,Y)-g_C(X,Y)^2}=\frac{3n(n-2)}{8(n-1)}\frac{\alpha\beta^4}{\alpha+n\beta}\left(\frac{2\alpha+n\beta}{2\alpha(\alpha+n\beta)+n\beta^2}\right)^2.
\end{equation*}

When $\beta\to-\frac{1}{n-1}$ (and $\alpha=1-\beta\to\frac{n}{n-1}$), we have $\alpha+n\beta\to0$ and we have $\alpha\beta^4\left(\frac{2\alpha+n\beta}{2\alpha(\alpha+n\beta)+n\beta^2}\right)^2\to\frac{n}{(n-1)^3}$ so this quantity tends to $+\infty$. Finally with $C=(1-\rho)I_n+\rho\mathds{11}^\top\in\Cor^+(n)$ with $\rho\in(-\frac{1}{n-1};1)$, $X=I_n-\mathds{11}^\top\in T_C\Cor^+(n)$ and $Y=\mu\mathds{1}^\top+\mathds{1}\mu^\top-2\diag(\mu)\in T_C\Cor^+(n)$, we have:
\begin{equation}
    \kappa_C(X,Y)\gs-\frac{1}{2}+\frac{3}{8}\frac{\mathds{1}^\top D(I_n+C\bullet C^{-1})^{-1}D\mathds{1}}{g_C(X,X)g_C(Y,Y)-g_C(X,Y)^2}\underset{\rho\to-\frac{1}{n-1}}{\lto} +\infty.
\end{equation}
This proves that the quotient-affine sectional curvature is not bounded from above.

\section{Proof of Theorem 4.9}

\begin{enumerate}
    \item The formula of the geodesic is known for the quotient-affine metrics in dimension 2. Hence it suffices to show that the quotient-affine metrics and the poly-hyperbolic-Cholesky metrics coincide up to a scaling factor. Let $C=C(\rho)$ and $X=\begin{pmatrix}0 & x\\x & 0\end{pmatrix}\in T_C\Cor^+(2)$. We compute the quotient-affine metric $g^{\mathrm{QA}}_C(X,X)=\tr(C^{-1}XC^{-1}X)-2\,\somme(D(I_n+C\bullet C^{-1})^{-1}D)$ and the canonical PHC metric $g^{\mathrm{CPHC}}_C(X,X)=\|\Diag(L)^{-1}L\,\low_S(L^{-1}XL^{-\top})\|^2$ where $D=\Diag(C^{-1}X)$ and $L=\Chol(C)$.
    \begin{align*}
        C^{-1}&=\frac{1}{1-\rho^2}\begin{pmatrix}1 & -\rho\\-\rho & 1\end{pmatrix},\\
        C^{-1}X&=\frac{x}{1-\rho^2}\begin{pmatrix}-\rho & 1\\1 & -\rho\end{pmatrix},
    \end{align*}
    \begin{align*}
        C^{-1}XC^{-1}X&=\frac{x^2}{(1-\rho^2)^2}\begin{pmatrix}1+\rho^2 & -2\rho\\-2\rho & 1+\rho^2\end{pmatrix},\\
        \tr(C^{-1}XC^{-1}X)&=\frac{2(1+\rho^2)}{(1-\rho^2)^2}x^2,\\
        I_n+C\bullet C^{-1}&=\frac{1}{1-\rho^2}\begin{pmatrix}2-\rho^2 & -\rho^2\\-\rho^2 & 2-\rho^2\end{pmatrix},\\
        (I_n+C\bullet C^{-1})^{-1}&=\frac{1}{4}\begin{pmatrix}2-\rho^2 & \rho^2\\\rho^2 & 2-\rho^2\end{pmatrix},\\
        D=\Diag(C^{-1}X)&=-\frac{\rho x}{1-\rho^2}I_2,\\
        \somme(D(I_n+C\bullet C^{-1})^{-1}D)&=\frac{\rho^2x^2}{(1-\rho^2)^2},\\
        g_C(X,X)&=\frac{2x^2}{(1-\rho^2)^2},
    \end{align*}
    \begin{align*}
        L=\Chol(C)&=\begin{pmatrix}1 & 0\\\rho & \sqrt{1-\rho^2}\end{pmatrix},\\
        L^{-1}XL^{-\top}&=\begin{pmatrix}1 & 0\\-\frac{\rho}{\sqrt{1-\rho^2}} & \frac{1}{\sqrt{1-\rho^2}}\end{pmatrix}\begin{pmatrix}0 & x\\x & 0\end{pmatrix}\begin{pmatrix}1 & -\frac{\rho}{\sqrt{1-\rho^2}}\\0 & \frac{1}{\sqrt{1-\rho^2}}\end{pmatrix}\\
        &=\begin{pmatrix}0 & x\\ \frac{x}{\sqrt{1-\rho^2}} & -\frac{\rho x}{\sqrt{1-\rho^2}} \end{pmatrix}\begin{pmatrix}1 & -\frac{\rho}{\sqrt{1-\rho^2}}\\0 & \frac{1}{\sqrt{1-\rho^2}}\end{pmatrix}\\
        &=\begin{pmatrix}0 & \frac{x}{\sqrt{1-\rho^2}}\\\frac{x}{\sqrt{1-\rho^2}} & -\frac{2\rho x}{1-\rho^2}\end{pmatrix},
    \end{align*}
    \begin{align*}
        \Diag(L)^{-1}L\,\low_S(L^{-1}XL^{-\top})&=\begin{pmatrix}1 & 0\\\frac{\rho}{\sqrt{1-\rho^2}} & 1\end{pmatrix}\begin{pmatrix}0 & 0\\\frac{x}{\sqrt{1-\rho^2}} & -\frac{\rho x}{1-\rho^2}\end{pmatrix}\\
        &=\begin{pmatrix}0 & 0\\\frac{x}{\sqrt{1-\rho^2}} & -\frac{\rho x}{1-\rho^2}\end{pmatrix},\\
        g^{\mathrm{CPHC}}_C(X,X)&=\left(\frac{1}{1-\rho^2}+\frac{\rho^2}{(1-\rho^2)^2}\right)x^2\\
        &=\frac{x^2}{(1-\rho^2)^2}.
    \end{align*}
    
    \item Euclidean-Cholesky and log-Euclidean-Cholesky metrics coincide in dimension 2 because $\exp(\xi)=I_n+\xi$ and $\log(\Gamma)=\Gamma-I_2$ for $\xi\in\LT^0(2)$ and $\Gamma\in\LT^1(2)$. Thus their common Riemannian exponential and logarithm in $\LT^1(2)$ are simply $\Exp_\Gamma(\xi)=\Gamma+\xi$ and $\Log_\Gamma(\Gamma')=\Gamma'-\Gamma$. On the other hand, the group exponential in $\LT^1(2)$ is $\Exp^{\LT^1(2)}_\Gamma(\xi)=\Gamma\exp(\Gamma^{-1}\xi)=\Gamma(I_n+\Gamma^{-1}\xi)=\Gamma+\xi=\Exp_\Gamma(\xi)$. Hence, the group geodesics coincide with the (log-)Euclidean-Cholesky geodesics. Let us compute them.
    \begin{align*}
        \Gamma_1&:=\Diag(\Chol(C_1))^{-1}\Chol(C_1)=\begin{pmatrix}1 & 0\\\frac{\rho_1}{\sqrt{1-\rho_1^2}} & 1\end{pmatrix},\\
        \Gamma_2&:=\Diag(\Chol(C_2))^{-1}\Chol(C_2)=\begin{pmatrix}1 & 0\\\frac{\rho_2}{\sqrt{1-\rho_2^2}} & 1\end{pmatrix},\\
        C(t)&=\Theta^{-1}((1-t)\Gamma_1+t\Gamma_2)=\Theta^{-1}\begin{pmatrix}1 & 0\\F(t) & 1\end{pmatrix}\\
        &=\cor\begin{pmatrix}1 & F(t)\\F(t) & 1+F(t)^2\end{pmatrix}=\begin{pmatrix}1 & \frac{F(t)}{\sqrt{1+F(t)^2}}\\\frac{F(t)}{\sqrt{1+F(t)^2}} & 1\end{pmatrix}.
    \end{align*}
    We can also compute the (log-)Euclidean-Cholesky metric in dimension 2.
    \begin{align*}
        g_C(X,X)&=\|d_C\Theta(X)\|^2=\|(\Theta\circ C)'(\rho)\|x^2\\
        &=f'(\rho)^2x^2,
    \end{align*}
    where $f(\rho)=\frac{\rho}{\sqrt{1-\rho^2}}$. So $f'(\rho)=\frac{\sqrt{1-\rho^2}+\frac{\rho^2}{\sqrt{1-\rho^2}}}{1-\rho^2}=\frac{1}{(1-\rho^2)^{3/2}}$ and:
    \begin{equation*}
        g_C(X,X)=\frac{x^2}{(1-\rho^2)^3}.
    \end{equation*}
\end{enumerate}

\end{appendices}

\bibliographystyle{siamplain}
\bibliography{references}

\end{document}

%% file: shared.tex
\usepackage{amsfonts}
\usepackage{graphicx}
\usepackage{epstopdf}
\usepackage{algorithmic}

\usepackage{amssymb}
\usepackage{lineno}
\usepackage{hyperref}
\usepackage{amsmath}
\usepackage{stmaryrd}
\usepackage{multirow}
\usepackage{enumitem}
\usepackage{dsfont}
\usepackage{makecell}
\usepackage{soul}
\usepackage{appendix}
\usepackage{multirow}
\usepackage{arydshln}
\setuldepth{Paris}

\ifpdf
  \DeclareGraphicsExtensions{.eps,.pdf,.png,.jpg}
\else
  \DeclareGraphicsExtensions{.eps}
\fi

\newcommand{\R}{\mathbb{R}}

\newcommand{\Sph}{\mathbb{S}}
\newcommand{\Hyp}{\mathbb{H}}

\newcommand{\mc}{\mathcal}
\newcommand{\mf}{\mathfrak}
\newcommand{\lto}{\longrightarrow}
\newcommand{\lmto}{\longmapsto}

\newcommand{\nec}{\Longrightarrow}

\newcommand{\ls}{\leqslant}
\newcommand{\gs}{\geqslant}

\newcommand{\tr}{\mathrm{tr}}

\newcommand{\vecto}{\mathrm{vec}}

\newcommand{\diag}{\mathrm{diag}}
\newcommand{\Diag}{\mathrm{Diag}}

\newcommand{\somme}{\mathrm{sum}}
\newcommand{\Id}{\mathrm{Id}}

\newcommand{\Chol}{\mathrm{Chol}}

\newcommand{\Mat}{\mathrm{Mat}}
\newcommand{\Sym}{\mathrm{Sym}}

\newcommand{\Skew}{\mathrm{Skew}}

\newcommand{\GL}{\mathrm{GL}}
\newcommand{\Orth}{\mathrm{O}}
\newcommand{\SO}{\mathrm{SO}}
\newcommand{\SL}{\mathrm{SL}}

\newcommand{\LT}{\mathrm{LT}}
\newcommand{\UT}{\mathrm{UT}}
\newcommand{\Cov}{\mathrm{Cov}}
\newcommand{\Cor}{\mathrm{Cor}}
\newcommand{\cor}{\mathrm{cor}}

\newcommand{\Exp}{\mathrm{Exp}}
\newcommand{\Log}{\mathrm{Log}}

\newcommand{\dotprod}[2]{\langle #1|#2\rangle}

\newcommand{\low}{\mathrm{low}}
\newcommand{\off}{\mathrm{off}}
\newcommand{\ver}{\mathrm{ver}}
\newcommand{\hor}{\mathrm{hor}}

\newcommand{\Hol}{\mathrm{Hol}}
\newcommand{\QA}{\mathrm{QA}}
\newcommand{\ad}{\mathrm{ad}}
\renewcommand{\S}{\mathrm{S}}
\newcommand{\CAT}{\mathrm{CAT}}

\newcommand{\fun}[4]{
\left\{
\begin{array}{ccc}
#1 & \lto & #2\\ \relax
#3 & \lmto & #4\\
\end{array}
\right.
}

\newsiamremark{remark}{Remark}
\newsiamremark{hypothesis}{Hypothesis}
\crefname{hypothesis}{Hypothesis}{Hypotheses}
\newsiamthm{claim}{Claim}

\headers{Geometries on correlation matrices}{Y. Thanwerdas and X. Pennec}

\title{Theoretically and computationally convenient geometries on full-rank correlation matrices \thanks{Submitted to the editors January 14th, 2022.}}

\author{Yann Thanwerdas\thanks{Université Côte d'Azur and Inria, Epione Project Team (\email{yann.thanwerdas@inria.fr}).}
\and Xavier Pennec\thanks{Université Côte d'Azur and Inria, Epione Project Team (\email{xavier.pennec@inria.fr}).}
}

\usepackage{amsopn}
